\documentclass[11pt]{article}

\usepackage{fancyhdr}
\usepackage{amsfonts}
\usepackage{amsmath}
\usepackage{amssymb}
\usepackage{amsthm}
\usepackage{tikz}
\usepackage{cases}

\usepackage{amscd}
\usepackage{amstext}

\usepackage{graphics}
\usepackage{graphicx}

\newlength{\fixboxwidth}
\newcommand{\re}{\mathbb{R}}\newcommand{\N}{\mathbb{N}}
\newcommand{\zz}{\mathbb{Z}}
\newcommand{\tor}{\mathbb{T}}\newcommand{\T}{\tor^d}

\newcommand{\Z}{{\zz}^d}

\newcommand{\R}{{\re}^d}

\newcommand{\cl}{{\mathcal L}}

\newcommand{\ce}{{\mathcal E}}

\newcommand{\supp}{{\rm supp \, }}

\newcommand{\gf}{\mathcal{F}}

\newcommand{\bproof}{\begin{proof}}
\newcommand{\eproof}{\end{proof}}

\newcommand{\be}{\begin{equation}}
\newcommand{\ee}{\end{equation}}
\newcommand{\beq}{\begin{eqnarray}}
\newcommand{\beqq}{\begin{eqnarray*}}
\newcommand{\eeq}{\end{eqnarray}}
\newcommand{\eeqq}{\end{eqnarray*}}

\numberwithin{equation}{section}

\newtheorem{theorem}{Theorem}[section]
\newtheorem{definition}[theorem]{Definition}
\newtheorem{corollary}[theorem]{Corollary}
\newtheorem{lemma}[theorem]{Lemma}
\newtheorem{proposition}[theorem]{Proposition}
\newtheorem{remark}[theorem]{Remark}

\begin{document}


\title{Weyl and Bernstein Numbers of Embeddings of Sobolev Spaces with Dominating Mixed Smoothness}

\author{Van Kien Nguyen\\
Friedrich-Schiller-University Jena, Ernst-Abbe-Platz 2, 07737 Jena, Germany\\
University of Transport and Communications, Lang Thuong, Dong Da, Hanoi, Vietnam\\
Email: kien.nguyen@uni-jena.de
}


\date{\today}

\maketitle

\begin{abstract}This paper is a continuation of the papers \cite{KiSi} and \cite{Ki}. Here we shall investigate the asymptotic behaviour of Weyl and Bernstein numbers of embeddings of Sobolev spaces with dominating mixed smoothness into Lebesgue spaces.\end{abstract}

\section{Introduction}\label{sec-intro}
Weyl numbers have been introduced by Pietsch \cite{Pi80-2}. Let $X$, $Y$ be Banach spaces. The $n$th Weyl number of the linear operator $T \in \mathcal L(X,Y)$ is given by 
$$ x_n(T):=\sup\{a_n(TA):\ A\in \mathcal L(\ell_2,X),\ \|A\|\leq 1\}\, , \qquad n \in \N\, .$$
Here $a_n(TA)$ is the $n$th approximation number of the operator $TA$. Recall, the $n$th approximation number of the linear operator $T \in \mathcal L(X,Y)$ is defined to be 
$$ a_n(T):=\inf\{\|T-A:X\to Y\|: \ A\in \mathcal L(X,Y),\ \ \text{rank} (A)<n\}\, , \qquad n \in \N\, . $$
The particular interest in Weyl numbers stems from the fact that 
they are the smallest known $s$-numbers satisfying the famous Weyl-type inequalities. This is, if $T: \, X \to X $ is a compact linear operator in a Banach space $X$, then
\beqq
\Big( \prod_{k=1}^{2n-1} |\lambda_k (T)|\Big)^{1/(2n-1)} \le \sqrt{2e} \, \Big( \prod_{k=1}^{n} x_k (T)\Big)^{1/n}
\eeqq
holds for all $n \in \N$, in particular, 
\[
|\lambda_{2n-1} (T)| \le \sqrt{2e} \, \Big( \prod_{k=1}^{n} x_k (T)\Big)^{1/n} \, ,
\]
see Pietsch \cite{Pi80-2} and Carl, Hinrichs \cite{CH}.
Here 
$(\lambda_n(T))_{n=1}^\infty$ is the sequence of non-zero eigenvalues of $T$, ordered in the following way: each eigenvalue is repeated according to its algebraic multiplicity and $|\lambda_n(T)|\geq |\lambda_{n+1}(T)|$, $n\in \N$. Hence, Weyl numbers may be seen as an appropriate tool to control the eigenvalues of $T$.

The behaviour of Weyl numbers has been considered at various places since 1980, for example, Pietsch \cite{Pi80-1,Pi80-2}, Lubitz \cite{Claus}, K\"onig \cite{Koe} and Caetano \cite{Cae1,Cae2,Caed}. They studied Weyl numbers of embeddings $id: B^t_{p_1,q_1}((0,1)^d) \to L_{p_2}((0,1)^d)$, where
$B^t_{p_1,q_1}((0,1)^d)$ denotes the isotropic Besov spaces defined on $(0,1)^d$.
Zhang, Fang, Huang \cite{Fang1} and Gasiorowska, Skrzypczak \cite{GS} investigated the case of embeddings of weighted Besov spaces, defined on $\R$, into Lebesgue spaces. In addition we refer to \cite{KiSi}, where the authors investigated the order of Weyl numbers with respect to embeddings of tensor product Besov spaces. 

Bernstein numbers were introduced  by Mityagin and Pelczy\'nski \cite[page 370]{MiPe63}. Recall that the $n$th Bernstein number of $T\in \mathcal{L}(X,Y)$ is defined to be
$$ b_n(T)= \sup_{L_n}\inf_{\substack{x\in L_n
\\ x\not =0}} \dfrac{\|Tx\,|\,Y\|}{\| x\,|\,X\|} ,$$
where the supremum is taken over all subspaces $L_n$ of $X$ with dimension $n$. Bernstein numbers are well-known to be  lower bounds for nonlinear widths, Kolmogorov and Gelfand numbers, see \cite{Devo2, Pi-74} and \cite[Theorem 3, page 190]{Ti2}. Weyl and Bernstein numbers are related. If $X,Y$ are two Banach spaces and $T\in\mathcal{L}(X,Y)$, then
\be\label{bern1} b_{2n-1}(T)\leq e\Big(\prod_{k=1}^{n}x_k(T)\Big)^{\frac{1}{n}} \ee
holds for all $n\in \N$, see \cite{Pi-08}. It is obvious that if $x_n(T)\asymp n^{-\alpha}(\log n)^{\beta}$, $\alpha, \beta\geq 0$, $n\geq 2$, then we  have $b_n(T)\leq Cx_n(T)$, for $n\in \N$. In this paper, we shall show that Bernstein numbers are also dominated by entropy numbers, i.e., $b_n(T)\leq 2\sqrt{2}e_n(T)$, $n\in \N$, see Lemma \ref{bern-en}. Hence, we have
\be\label{bnxnen}
b_n(T)\leq C \min\{x_n(T), e_n(T)\}\, ,
\ee
for all $n\in \N$ if $x_n(T)$ behaves polynomially.

Bernstein numbers do not have so "nice" properties as Weyl numbers, see Section \ref{sec-pro} or Pietsch \cite{Pi-08}. This is, may be, the reason why the picture concerning the behaviour of Bernstein numbers is less complete than in the case of Weyl numbers. In the literature, the order of Bernstein numbers were studied in different situations. In the one-dimensional periodic situation, the behaviour of Bernstein numbers of the embeddings of Sobolev spaces into Lebesgue spaces was calculated by Tsarkov and Maiorov, see \cite[Theorem 12, page 194]{Ti2}.    
 Galeev in \cite{Gale1} studied the behaviour of Bernstein numbers of embeddings $S^{t}_{p_1}H(\T)\to L_{p_2}(\T)$ and $S^{t}_{p_1,\infty}B(\T)\to L_{p_2}(\T)$. Here $S^{t}_{p_1}H(\T)$ and $S^{t}_{p_1,\infty}B(\T)$ are Sobolev and Nikol'skij spaces of dominating mixed smoothness on the d-dimensional torus $\T$. The picture given in \cite{Gale1} was not complete, e.g., in the cases of low smoothness.

 Denote $\Omega= (0,1)^d$. In this paper we shall give the complete picture, up to some limiting cases,  of the behaviour of Weyl and Bernstein numbers of the embeddings $id: S^t_{p_1}H(\Omega)\to L_{p_2}(\Omega)$. The method we apply here for Weyl numbers could be called standard compared to Vybiral \cite{Vy} or \cite{KiSi}. Because of the polynomial behaviour of $x_n(id: S^t_{p_1}H(\Omega)\to L_{p_2}(\Omega))$, see Theorem \ref{main1}, by taking into account the inequality \eqref{bnxnen} we can obtain the upper bound for Bernstein numbers. In fact we shall show that the inequality \eqref{bnxnen} is the sharp estimate, i.e., 
\beqq 
b_n(id: S^t_{p_1}H(\Omega)&\to& L_{p_2}(\Omega))\\
&\asymp& \min\{x_n(id: S^t_{p_1}H(\Omega)\to L_{p_2}(\Omega))\,,\ e_n(id: S^t_{p_1}H(\Omega)\to L_{p_2}(\Omega))\}\,,
\eeqq
for all $n\in \N$.

The paper is organized as follows. Our main results are discussed in Section \ref{sec-main}. 
In Section \ref{sec-pro} we recall the definition and some properties of Weyl and Bernstein numbers.
Section \ref{sec-fun} is devoted to the function spaces under consideration.
The heart of the paper is Section \ref{sec-seq} in which we prove the behaviour of Weyl and Bernstein numbers of embeddings of certain sequence spaces
associated to spaces of dominating mixed smoothness.
In Section \ref{sec-proof} we shift the results in sequence spaces to the situation of function spaces. Here our main results will be  proved.


\subsection*{Notation}


As usual, $\N$ denotes the natural numbers, $\N_0 := \N \cup \{0\}$,
$\zz$ the integers and
$\re$ the real numbers. For a real number $a$ we put $a_+ := \max(a,0)$.
By $[a]$ we denote the integer part of $a\in \re$.
If $\bar{j}=(j_1, \ldots \, , j_d)\in \N_0^d$,   then we put $ |\bar{j}|_1 := j_1 + \ldots \, + j_d\, .
$ By $\Omega$ we denote the unit cube in $\R$, i.e., $\Omega:= (0,1)^d$.
As usual, the symbol $c $ denotes positive constants 
which depend only on the fixed parameters $t,p,q$ and probably on auxiliary functions, unless otherwise stated; its value may vary from line to line.
The symbol $A \lesssim B$ indicates that there exists a constant $c>0$ such that
 $A \le c \,B$. Similarly $\gtrsim$ is defined. The symbol 
$A \asymp B$ will be used as an abbreviation of
$A \lesssim B \lesssim A$.
For a discrete set $D$ the symbol $|D|$ 
denotes the cardinality of this set.
The symbol $id_{p_1,p_2}^m$ refers to the identity 
\beqq
id_{p_1,p_2}^m:~ \ell_{p_1}^m \to \ell_{p_2}^m\, .
\eeqq
Finally, if we write $\omega_n $, we mean   either $x_n$ or $b_n$.
\section{The main results}\label{sec-main}
First, let us recall that the embedding $id: S^t_{p_1}H(\Omega)\to L_{p_2}(\Omega)$, $1<p_1<\infty,\ 1\leq p_2\leq\infty$, is compact if and only if $ t>(\frac{1}{p_1}-\frac{1}{p_2})_+$, see Vybiral \cite[Theorem 3.17]{Vy}. Since we are exclusively interested in compact embeddings this condition is always present throughout the paper. 
\begin{theorem}\label{main1}
Let $1< p_1,p_2< \infty $ and $ t>(\frac{1}{p_1}-\frac{1}{p_2})_+$. Then we have
$$x_n(id: S^t_{p_1}H(\Omega)\to L_{p_2}(\Omega))\asymp n^{-\alpha}(\log n)^{(d-1)\alpha}\, , \ \ \ n\geq 2,$$
where
\begin{enumerate}
\item {\makebox[3.5cm][l]{$\alpha=t$} if\ \  $ p_1, p_2\leq 2$};
\item {\makebox[3.5cm][l]{$\alpha=t-\frac{1}{2}+\frac{1}{p_2}$} if\ \  $  p_1\leq 2\leq p_2 $};
\item {\makebox[3.5cm][l]{$\alpha=t-\frac{1}{p_1}+\frac{1}{2}$} if\ \  $  p_2\leq 2< p_1   $,\ $t>\frac{1}{p_1}$};
\item {\makebox[3.5cm][l]{$\alpha=t-\frac{1}{p_1}+\frac{1}{p_2}$} if\ \  $2< p_1\leq p_2 $ or $\big(2\leq p_2< p_1\, ,  \ t>\frac{{1}/{p_2}-{1}/{p_1}}{{p_1}/{2}-1}\big)$};
\item {\makebox[3.5cm][l]{$\alpha=\frac{tp_1}{2}$} if\ \ $\big(2\leq p_2<p_1\, ,\ t<\frac{{1}/{p_2}-{1}/{p_1}}{{p_1}/{2}-1}\big)$ or  $\big(p_2\leq 2< p_1\, ,  t<\frac{1}{p_1}\big)$}.
\end{enumerate}
\end{theorem}
Our results for Bernstein numbers read as follows.

\begin{theorem}\label{main2}
Let $1< p_1,p_2< \infty $ and $ t>(\frac{1}{p_1}-\frac{1}{p_2})_+$. Then we have
$$b_n(id: S^t_{p_1}H(\Omega)\to L_{p_2}(\Omega))\asymp  n^{-\beta}(\log n)^{(d-1)\beta}\, ,\ \ \ n\geq 2,$$
where
\begin{enumerate}
\item {\makebox[3.5cm][l]{$\beta=t$} if\ \  $ p_1\leq  p_2$ or $p_2\leq p_1\leq 2$};
\item {\makebox[3.5cm][l]{$\beta=t-\frac{1}{p_1}+\frac{1}{2}$} if\ \  $  p_2\leq 2< p_1   $,\ $t>\frac{1}{p_1}$};
\item {\makebox[3.5cm][l]{$\beta=t-\frac{1}{p_1}+\frac{1}{p_2}$} if\ \ $2\leq p_2< p_1  $,\ $t>\frac{{1}/{p_2}-{1}/{p_1}}{{p_1}/{2}-1}$};
\item {\makebox[3.5cm][l]{$\beta=\frac{tp_1}{2}$} if\ \ $\big(2\leq p_2<p_1\, ,\ t<\frac{{1}/{p_2}-{1}/{p_1}}{{p_1}/{2}-1}\big)$ or  $\big(p_2\leq 2< p_1\, ,  t<\frac{1}{p_1}\big)$}.
\end{enumerate}
\end{theorem}
\begin{remark}\rm
{\rm (i)} Theorems \ref{main1} and \ref{main2} give the final answer about the behaviour of Weyl and Bernstein numbers of embeddings $id: S^t_{p_1}H(\Omega)\to L_{p_2}(\Omega)$ in all cases, except for some limiting cases.\\
{\rm (ii)} The results in Theorem \ref{main2} should be compared with the results of  Galeev in  \cite{Gale1}. The cases (i), (ii) and (iii) are also considered by Galeev. However he was using some additional conditions in smoothness. Galeev  \cite{Gale1} was unable to determine the asymptotic behaviour of 
Bernstein numbers in the cases of low smoothness (iv).
\end{remark}
\begin{remark}\rm
{\rm ( i)} It is interesting that the power of $n$ and the power of $\log n$ coincide in both Theorems \ref{main1} and \ref{main2}.\\
{\rm (ii)} Surprisingly, Theorems \ref{main1} and \ref{main2} show that
\beqq 
b_n(id: S^t_{p_1}H(\Omega)&\to& L_{p_2}(\Omega))\\
&\asymp& \min\{x_n(id: S^t_{p_1}H(\Omega)\to L_{p_2}(\Omega))\,,\ e_n(id: S^t_{p_1}H(\Omega)\to L_{p_2}(\Omega))\}.
\eeqq
Here $e_n$ is the $n$th entropy number, see definition in Section \ref{sec-pro}. For the behaviour of $e_n(id: S^t_{p_1}H(\Omega)\to L_{p_2}(\Omega))$ we refer to \cite{Be,DD,Tem2,Vy}.
\end{remark}
By the abstract properties of Weyl numbers we can extend Theorem \ref{main1} to the following extreme situations.
\begin{theorem}\label{main3}
Let $1<p< \infty$. Then we have
\begin{enumerate}
\item \beqq
x_{n} (id: \ S^{t}_{p}H(\Omega) \to L_\infty(\Omega)) \asymp
\left\{
\begin{array}{lll}
n^{-t+\frac{1}{2}} (\log n)^{(d-1)t}  & \mbox{if}\quad p \le 2\,, \ t>\frac{1}{p}, 
\\ 
n^{-t+\frac{1}{p}} (\log n)^{(d-1)(t- \frac{1}{p} + \frac 12 )}  &  \mbox{if}\quad 2 < p\,, \ t>\frac{1}{p}+\frac{1}{2},
\end{array}
\right.
\eeqq
\item 
and
\beqq
x_n(id: ~S^t_{p}H(\Omega)\to L_1(\Omega))
\asymp 
\left\{\begin{array}{lll}
n^{-t}(\log n)^{(d-1)t}  &  \text{if}\ \ p\leq 2\, ,\ t>0, \\
n^{-t+\frac{1}{p}-\frac{1}{2}} (\log n)^{(d-1)(t-\frac{1}{p}+\frac{1}{2})}  & \text{if}\ \  2<p,\ t> \frac{1}{p},\\
n^{-\frac{tp}{2}} (\log n)^{(d-1)\frac{tp}{2}}  & \text{if}\ \  2<p,\ t< \frac{1}{p},\
\end{array}\right.
\eeqq
for all $n \ge 2$.
\end{enumerate}

\end{theorem}
\begin{remark}\label{rem}
\rm
{\rm (i)} The behaviour of $x_n\big(id: \ S^{t}_{p}H(\Omega) \to L_\infty(\Omega)\big)$ in case $2< p,\ t\in \big(\frac{1}{p},\frac{1}{2}+\frac{1}{p}\big]$ is open.\\
{\rm (ii)} Beside the result of Temlyakov for Kolmogorv numbers, see \cite{Te96},
\beqq
d_{n} (id: S^{t}_{p}H(\tor^2) \to L_\infty(\tor^2)) \asymp n^{-t}\, (\log n)^{t+\frac{1}{2}} \, , \qquad p\geq 2\,,\ \ t>\frac{1}{2}, 
\eeqq
 and approximation numbers, see \cite{Te93} and also Cobos, K\"uhn, Sickel \cite{CKS},
\be\label{app}
 a_{n} (id: \, S^{t}_{2}H (\T) \to L_\infty (\T)) \asymp
n^{-t+\frac 12}(\log n)^{(d-1)t}\, ,\qquad t>\frac{1}{2} \,,\ \ n\geq 2\, ,
\ee
we are not aware of any other result giving the exact order of $s$-numbers of $id: \ S^{t}_{p}H(\Omega) \to L_\infty(\Omega) $. Note that \eqref{app} also holds true for Gelfand and Weyl numbers since
\be\label{an-xn}
x_n(T:\ H\to Y)\ =\ c_n(T:\ H\to Y)\ =\ a_n(T:\ H\to Y)
\ee
if $H$ is a Hilbert space and $T\in \mathcal{L}(H,Y)$, see \cite[Proposition 2.4.20]{Pi-87}.
\end{remark}

\section{Weyl and Bernstein numbers - Properties}\label{sec-pro}
Weyl numbers are special $s$-numbers. Let $X,Y,X_0,Y_0$ be Banach spaces.
 An $s$-function is a map $s$ assigning to every operator $T\in \mathcal L(X,Y)$ a scalar sequence $\{s_n(T)\}_{n\in \N}$ such that the following conditions are satisfied:
\begin{enumerate}
\item[(a)] $\|T\|=s_1(T)\geq s_2(T)\geq...\geq 0 $;
\item[(b)] $ s_{n+m-1}(S+T)\leq s_n(S)+ s_m(T) $ for all $S\in \mathcal L(X,Y)$ and $m,n=1,2, \ldots \, $;
\item[(c)] $s_n(BTA)\leq \|B\| \, \cdot \, s_n(T) \, \cdot \, \|A\|$ for all $A\in \mathcal L(X_0,X)$, $B\in \mathcal L(Y,Y_0)$;
\item[(d)] $s_n(T)=0$ if $\text{rank}(T)<n$ for all $n\in \N$; 
\item[(e)] $s_n(id: \ell_2^n\to \ell_2^n)=1$ for all $n\in \N$.
\end{enumerate}

An $s$-function is called multiplicative if
\begin{enumerate}
\item[(f)]$ s_{n+m-1}(ST)\leq s_n(S)\, s_m(T) $ for all $S\in \mathcal L(Y,Z)$ and $m,n=1,2, \ldots \, $.
\end{enumerate}
Let us recall some well-known $s$-numbers:
\begin{enumerate}
\item 
Approximation and Weyl numbers are multiplicative $s$-numbers, see  \cite[2.3.3]{Pi-87}.
\item The $n$th Kolmogorov number of the linear operator $T \in \mathcal L(X,Y)$ is defined to be 
\be
d_n(T)= \inf_{L_{n-1}}\sup_{\|x|X\|\leq 1}\inf_{y\in L_{n-1}}\|Tx-y|Y\|. \label{def1}
\ee
Here the outer supremum is taken over all linear subspaces  $L_{n-1}$ of dimension ($n-1$)  in $Y$. Kolmogorov numbers are multiplicative $s$-numbers, see, e.g., \cite[Theorem 11.9.2]{Pi80-2}.
\item The $n$th Gelfand number of the linear operator $T \in \mathcal L(X,Y)$ is defined to be 
$$ c_n (T) := \inf\Big\{\|\, T\, J_M^X\, \|: \ {\rm codim\,}(M)< n\Big\},$$
where $J_M^X:M\to X$ refers to the canonical injection of $M$ into $X$. 
Gelfand numbers are multiplicative $s$-numbers, see  \cite[Proposition 2.4.8]{Pi-87}.
\end{enumerate}

Weyl and Gelfand numbers share the common interpolation property, see \cite{KiSi,Tr70}.
\begin{proposition}\label{inter}
Let $0<\theta<1$.
Let $X,Y, Y_0,Y_1$ be Banach spaces. 
Further we assume  $Y_0\cap Y_1\hookrightarrow  Y$ and the existence of  a positive constant $C$ such that
\be\label{weylextra3}
\| y|Y\| \leq C \, \| y|Y_0\|^{1-\theta}\|y|Y_1\|^{\theta}\qquad  \text{for all}\quad  y\in Y_0\cap Y_1.
\ee
Then, if $
T\in  \mathcal{L}(X,Y_0) \cap  \mathcal{L}(X,Y_1) \cap  \cl (X,Y)
$
we obtain 
\beqq
x_{n+m-1}(T:~X\to Y)\le C\,  x_n^{1-\theta}(T:~ X\to Y_0)\, x_m^{\theta}(T:~X\to Y_1)
\eeqq
for all $n,m \in \N$. Here $C$ is the same constant as in \eqref{weylextra3}.
\end{proposition}
Next we shall discuss  some properties of Bernstein numbers. It is obvious that Bernstein numbers satisfy (a), (c), (d) and (e). However, they are not $s$-numbers because they fail to satisfy property (b), see \cite{Pi-08}. Bernstein numbers satisfy a weaker inequality than (b), namely
\begin{enumerate}
\item [(b')] $ b_{n}(S+T)\leq \|S\|+ b_n(T) $ for all $S, \,T\in \mathcal L(X,Y)$ and $n\in \N$.
\end{enumerate}
Bernstein numbers are also not multiplicative, see again \cite{Pi-08}. It is proved that Bernstein numbers are dominated by Gelfand and Kolmogorov numbers, see \cite{Pi-74}. In some special cases Bernstein numbers are bounded by Weyl numbers. 
\begin{lemma}
\label{bern-weyl0}
Let $X,Y$ be two Banach spaces, $T \in\mathcal{L}(X,Y)$ and $\alpha,\beta\in \mathbb{R},\ \beta\geq 0$. Assume that $x_n(T)\asymp n^{-\alpha}(\log n)^{\beta}$, $n\geq 2$. Then
$$b_n(T)\lesssim x_n(T) ,$$ 
for all $n\in \mathbb{N}$. Moreover, if $Y$ is a Hilbert space we have $b_n(T)\asymp x_n(T)$.
\end{lemma}
The proof of  Lemma \ref{bern-weyl0} is given in \cite{Ki}. It is based on the inequality \eqref{bern1}, see \cite{Pi-08}. Next we consider the relation between  Bernstein and entropy numbers. The $n$th (dyadic) entropy
number of $T \in\mathcal{L}(X,Y)$ is defined as
$$
e_n (T):=\inf\{ \varepsilon
>0: T(B_X) \text{ can be covered by } 2^{n-1}
\text{ balls in } Y \text{ of radius } \varepsilon\}\, ,
$$
where $B_X:= \{x \in X: \: \|x|X\| \le 1\}$ denotes the closed
unit ball of $X$. Note that entropy numbers are not $s$-numbers. We have the following lemma.
\begin{lemma}\label{bern-en}Let $X,Y$ be Banach spaces and $T\in \mathcal{L}(X,Y)$. Then we have
\beqq
b_{n}(T)\ \leq\ 2\sqrt{2}\, e_n(T),\ \ \ n\geq 1.
\eeqq
\end{lemma}
\begin{proof} Without loss of generality we assume that $b_n(T)>0$. Then for every $\epsilon>0$, $\epsilon<b_n(T)$, there exists a linear subspace $L_n$ of dimension $n$ in $X$ such that
\beqq
0<b_n(T)-\epsilon \leq \frac{\|Tx\,|\,Y\|}{\|x\,|\,X\|}\, ,\ \ \forall x\in L_n.
\eeqq
Denote by $E$ the canonical embedding of $L_n$ into $X$.  Then $TE$ induces an isomorphism $S$ between $L_n$ and $F_n:=TE(L_n)$. It is obvious that $\|S^{-1}:F_n\to L_n\|\leq (b_n(T)-\epsilon)^{-1}$. By $J$ we denote the canonical embedding from $F_n$ into $Y$. Let us consider the  diagram
\beqq
\begin{CD}
X  @ > T >> Y \\
 @A E AA @A J AA\\
L_n @ > S >> F_n. 
\end{CD}
\eeqq
By $\lambda_{n}(id: L_{n}\to L_{n})$ we denote the $n$th eigenvalue of the identity $id: L_{n}\to L_{n}$. The Carl-Triebel inequality, see \cite{CT}, and abstract property of entropy numbers yield
\beqq
1=\lambda_{n}(id: L_{n}\to L_{n})
\ \leq\ \sqrt{2}\, e_n(id: L_{n}\to L_{n})\ \leq \ \sqrt{2}\,\|S^{-1}\|\, e_n(S).
\eeqq
Because $J$ is an injection we have $e_n(S)\leq 2e_n(JS)$, see \cite[page 14]{CS}. Consequently we obtain
\beqq
1\ \leq \ 2\sqrt{2}\,\| S^{-1}\|\,e_n(JS)\ = \ 2\sqrt{2}\,\| S^{-1}\|\,e_n(TE)
\ \leq\  2\sqrt{2}\,\| S^{-1}\|\,e_n(T).
\eeqq
This implies
\beqq
b_n(T)-\epsilon &\leq& 2\sqrt{2}\,e_n(T).
\eeqq
Letting $\epsilon\downarrow 0$ we finish the proof.
\end{proof}
\begin{remark}\rm
The proof given in Lemma \ref{bern-en} is similar to the proof of Lemma 2 in \cite{Pi-08}.
\end{remark}
For later use, let us recall a result proved in \cite{Pi-74}.
\begin{lemma}\label{bern-gel}
Let $X, Y$ be  Banach spaces and $\dim(X)=\dim(Y)=m$. If $T\in \mathcal{L}(X,Y)$  is invertible then
$$b_n(T)c_{m-n+1}(T^{-1})=1\, ,\ \ n\in \mathbb{N},\ n\leq m.$$
\end{lemma}
Finally we turn to the behaviour of Weyl and Bernstein numbers of the embeddings $id_{p_1,p_2}^m: \ell_{p_1}^m\to \ell_{p_2}^m$. We refer to  \cite{Cae1,Cae2,Koe,Claus,Fang1} for the behaviour of Weyl numbers and to \cite{Gale1,Ki} for Bernstein numbers.
\begin{lemma}\label{Weyl1}
\begin{enumerate} 
\item Let $m,n\in \N$ and $2n\leq m$. Then we have
\begin{subnumcases}{x_n(id_{p_1,p_2}^{m})\asymp}
1&\text{ if } $2\leq p_1\leq  p_2 \leq \infty$, \label{wth1}  \\
n^{\frac{1}{p_2}-\frac{1}{p_1}}&\text{ if } $1\leq p_1\leq  p_2\leq 2$,\label{wth2}   \\
n^{\frac{1}{2}-\frac{1}{p_1}}&\text{ if } $1\leq p_1\leq 2 \leq p_2\leq \infty$,\nonumber \\
 m^{\frac{1}{p_2}-\frac{1}{p_1}}  &\text{ if } $1\leq p_2<p_1\leq 2$. \label{wth4}
\end{subnumcases}
\item  Let $2\leq p_2<p_1\leq \infty$ and $k,n \in \N$, $k\geq 2$. Then   $x_n(id_{p_1,p_2}^{kn})\asymp 1$.
\item  Let $1\leq p_2\leq 2< p_1\leq \infty$ and $m,n\in \mathbb{N}$. Then $x_n(id_{p_1,p_2}^m)\gtrsim m^{\frac{1}{p_2}-\frac{1}{2}}$\ \ if\ \  $n\leq \frac{m}{2}$.
\end{enumerate}
\end{lemma}

\begin{lemma}\label{Bern1}
\begin{enumerate}
\item Let $1\leq p_1,p_2\leq \infty$ and $n\in \N$. It holds
\begin{subnumcases}{b_n(id_{p_1,p_2}^{2n}) \gtrsim}
1&\text{ if } $2\leq p_2\leq p_1$,\label{th1} \\
n^{\frac{1}{p_2}-\frac{1}{2}}&\text{ if } $ p_2\leq 2\leq p_1$,\label{th2} \\
n^{\frac{1}{p_2}-\frac{1}{p_1}}&\text{ if } $ p_1\leq p_2 \ \text{ or }\ p_2\leq p_1\leq 2$.\label{th3}
\end{subnumcases}
\item Let $1< p_2\leq \max(p_2,2)< p_1\leq \infty$ and $n,m\in \N$. It holds
\beqq
 b_n(id_{p_1,p_2}^m)\gtrsim m^{\frac{1}{p_2}-\frac{1}{p_1}},\ \ 1\leq n\leq \big[m^{\frac{2}{p_1}}\big].
 \eeqq
\end{enumerate}
\end{lemma}
\section{Function spaces of dominating mixed smoothness}\label{sec-fun}
\subsection{Definition}
Let us define the Sobolev spaces of dominating mixed smoothness.
\begin{definition}\rm Let $1<p<\infty$ and $t\in \re$. Then $S^t_{p}H(\R)$ is the collection of all $f\in S'(\R)$ such that
\beqq
\| f|S^t_{p}H(\R)\|=\Big\|\gf^{-1}\Big[\prod_{i=1}^{d}(1+\xi_i^2)^{\frac{t}{2}}\gf f(\xi)\Big](\cdot)\Big|L_p(\R)\Big\| 
\eeqq
is finite. Here $\xi=(\xi_1,\cdots,\xi_d)\in \R$.
\end{definition}
\begin{remark}\rm
{\rm (i)} It is obvious that if $f_j \in H^t_{p}(\re)$, $j=1, \ldots \, , d,$ then $\prod_{j=1}^d f_j (x_j) \in S^t_{p}H(\R)$ and 
\[
 \Big\| \, \prod_{j=1}^d f_j (x_j) \, \Big| S^t_{p}H(\R)\Big\| = \prod_{j=1}^d \|\, f_j \, |H^t_{p} (\re)\| \,  .  
\]
Here $H^t_{p}(\re)$ is the Sobolev space of fractional order $t$ defined on $\re$. For the tensor product structure of the space $S^t_{p}H(\R)$ we refer to \cite{SUt}. \\
{\rm (ii)} In the literature sometimes the notation $MW_{p}^t(\R)$ is used instead of $S^t_pH(\R)$.\\
{\rm (iii)} By $S^m_{p}W(\R)$, $m\in \N_0$, we denote the classical Sobolev spaces of dominating mixed smoothness:
\beqq
S^m_pW(\R):= \Big\{f\in L_p(\R): \|f|S^m_pW(\R)\|:=\sum_{|\bar{\alpha}|_{\infty} \leq m}\|D^{\bar{\alpha}}f|L_p(\R)\|<\infty\Big\}\, .
\eeqq 
Here $\bar{\alpha}=(\alpha_1,...,\alpha_d)\in \N_0^d$ and $|\bar{\alpha}|_{\infty}=\max_{i=1,...,d}|\alpha_i|$.  We have $S^m_pH(\R)=S^m_{p}W(\R)$ in the sense of equivalent norms, see \cite[Theorem 2.3.1]{ST}.
\end{remark}
Sobolev spaces $ S^t_pH(\R)$ represent special cases of the Triebel-Lizorkin spaces of dominating mixed smoothness $S^{t}_{p,q}F(\R)$, see Theorem \ref{equal}. Let us introduce the function spaces
    $S^{t}_{p,q}F(\R)$. For details we  mainly refer
    to \cite[Chapter 2]{ST} and \cite[Chapter 1]{Vy}. The reader who is interested in other descriptions of these spaces may consult \cite{Baz1, Baz2, U1}. Let   $\varphi_0(x)\in C_0^{\infty}({\re})$ with $\varphi_0(x) = 1$ on $[-1,1]$ and $\supp\varphi_0 \subset [-\frac{3}{2},\frac{3}{2}]$. For $j\in \N$ we define
      $$
         \varphi_j(x) = \varphi_0(2^{-j}x)-\varphi_0(2^{-j+1}x).
      $$
   For $\bar{k} = (k_1,...,k_d) \in {\N}_0^d$ the function
    $\varphi_{\bar{k}}(x) \in C_0^{\infty}(\R)$ is defined by
    $$
        \varphi_{\bar{k}}(x) := \varphi_{k_1}(x_1)\cdot...\cdot
         \varphi_{k_d}(x_d)\,,\quad x\in \R.
    $$
   
    \begin{definition}\label{sprd} \rm Let $t \in \re$, $0<p<\infty$ and $0<q\leq \infty$. Then $S^{t}_{p,q}F(\R)$ is the collection of all
         $f\in S'(\R)$ such that
\beqq
            \|f|S^{t}_{p,q}F(\R)\|  =
            \bigg\|\bigg(\sum\limits_{\bar{k}\in{\N}_0^d}2^{|\bar{k}|tq}
            |\gf^{-1}[\varphi_{\bar{k}}\gf f](\cdot)|^q\bigg)^{1/q}
            \bigg|L_p(\R)\bigg\| <\infty\,.
\eeqq
    \end{definition}

\subsection*{Spaces on $\Omega=(0,1)^d$}

For us it will be convenient to define spaces on $\Omega $ by restrictions.
We shall need the set $D'(\Omega)$, consisting of all complex-valued distributions on $\Omega$.

\begin{definition} \label{defomega}\rm
{\rm (i)} Let $1<p<\infty$ and $t\in \re$. Then $S^t_{p}H(\Omega)$ is the space of all $f\in D'(\Omega)$ such that there exists a distribution $g\in S^t_{p}H(\R)$ satisfying $f = g|_{\Omega}$. It is endowed with the quotient norm
   $$
      \|\, f \, |S^{t}_{p}H(\Omega)\| = \inf \Big\{ \|g|S^{t}_{p}H(\R)\|~:~ g|_{\Omega} =
      f \Big\}\,.
   $$   
{\rm (ii)} Let $ 0< p < \infty$, $0< q\leq \infty$ and $t\in \re$. Then
   $S^{t}_{p,q}F(\Omega)$ is the space of all $f\in D'(\Omega)$ such that there exists a distribution $g\in
   S^{t}_{p,q}F(\R)$ satisfying $f = g|_{\Omega}$. It is endowed with the quotient norm
   $$
      \|\, f \, |S^{t}_{p,q}F(\Omega)\| = \inf \Big\{ \|g|S^{t}_{p,q}F(\R)\|~:~ g|_{\Omega} =
      f \Big\}\,.
   $$   
\end{definition}
We have the following theorem.
\begin{theorem}\label{equal}
Let $1 < p< \infty$ and $t\in \re$. Then
\[
S^t_pH(\R)= S_{p,2}^{t} F (\R) \qquad \mbox{and}\qquad S^t_pH(\Omega)= S_{p,2}^{t} F (\Omega)
\]
in the sense of equivalent norms.
\end{theorem}
Theorem \ref{equal} can be found in \cite[Theorem 2.3.1]{ST}. It is a consequence of Littlewood-Paley assertion, i.e., $L_p(\R)= S_{p,2}^{0} F (\R) $, see Nikol'skij \cite[1.5.6]{Ni}, and lifting property of spaces of dominating mixed smoothness, see \cite[2.2.6]{ST}. In the following Subsections and Sections we shall deal with the spaces $S^{t}_{p,2}F(\Omega)$ instead of $S^t_pH(\Omega)$ and $L_p(\Omega)$, respectively.

\subsection{Sequence spaces related to function spaces of dominating mixed smoothness}\label{seq-def}
Let us introduce some sequence spaces. If $\bar{\nu} =(\nu_1, \ldots ,\nu_d) \in \mathbb{N}_0^d$ and $\bar{m}= (m_1, \ldots , m_d)\in \mathbb{Z}^d$, then we put
\[
Q_{\bar{\nu},\bar{m}} := \Big\{x\in \R:\ 2^{-\nu_\ell} \, m_\ell < x_\ell < 2^{-\nu_\ell}\, (m_\ell+1)\, , \: \ell = 1, \, \ldots \, , d\Big\} \, .
\]
By $\chi_{{\bar{\nu},\bar{m}}}(\cdot)$ we denote the characteristic function of $Q_{\bar{\nu},\bar{m}}$. 

\begin{definition}\label{sequence1} If $t\in \mathbb{R}$, $0<p<\infty$,  $0<q\leq \infty$ and
$$\lambda:=\lbrace \lambda_{\bar{\nu},\bar{m}}\in\mathbb{C}:\bar{\nu}\in \mathbb{N}_0^d,\ \bar{m}\in \mathbb{Z}^d \rbrace\, ,$$
then we define 
$$s_{p,q}^tf=\Big\lbrace\lambda: \| \lambda|s_{p,q}^tf\| =
\Big\|\Big(\sum_{\bar{\nu}\in \mathbb{N}_0^d}\sum_{\bar{m}\in \mathbb{Z}^d}|2^{|\bar{\nu}|_1 t}
\lambda_{\bar{\nu},\bar{m}}\chi_{\bar{\nu},\bar{m}}(\cdot) 
|^q\Big)^{\frac{1}{q}}\Big|L_p(\mathbb{R}^d)\Big\|<\infty \Big\rbrace$$
with the usual modification for $q=\infty$.
\end{definition}

Now we recall wavelet bases of Lizorkin$-$Triebel spaces of dominating mixed smoothness.
Let $N \in \N$. Then there exists $\psi_0, \psi_1 \in C^N(\re) $, compactly supported, 
\[
\int_{-\infty}^\infty t^m \, \psi_1 (t)\, dt =0\, , \qquad m=0,1,\ldots \, , N\, , 
\]
such that
$\{ 2^{j/2}\, \psi_{j,m}:\ j \in \N_0, \: m \in \zz\}$, where
\[
\psi_{j,m} (t):= \left\{ \begin{array}{lll}
\psi_0 (t-m) & \qquad & \mbox{if}\quad j=0, \: m \in \zz\, , 
\\
\sqrt{1/2}\, \psi_1 (2^{j-1}t-m) & \qquad & \mbox{if}\quad j\in \N\, , \: m \in \zz\, , 
            \end{array} \right.
\]
is an orthonormal basis in $L_2 (\re)$, see \cite{woj}. 
Consequently, the system
\[
\Psi_{\bar{\nu}, \bar{m}} (x) := \prod_{\ell=1}^d \psi_{\nu_\ell, m_\ell} (x_\ell)\, \qquad \bar{\nu} \in \N_0^d, \, \bar{m} \in \Z\, , 
\]
is a tensor product wavelet basis of $L_2 (\R)$. Vybiral \cite[Theorem 2.12]{Vy} has proved the following.

\begin{lemma}\label{wavelet}
Let $0< p<\infty$\,, $0<q \le\infty$ and $t\in \re$. There exists $N=N(t,p,q) \in \N$ such that the mapping 
\beqq
{\mathcal W}: \quad f \mapsto (2^{|\bar{\nu}|_1}\langle f, \Psi_{\bar{\nu},\bar{m}} \rangle)_{\bar{\nu} \in \N_0^d\, , \, \bar{m} \in \Z} 
\eeqq
is an isomorphism of $S^t_{p,q}F(\R)$ onto $s^t_{p,q}f$.
\end{lemma}

We put
\beqq
A_{\bar{\nu}}^{\Omega}:= 
\Big\{\bar{m}\in \mathbb{Z}^d:\ \supp \Psi_{\bar{\nu},\bar{m}} \cap\Omega \neq \emptyset\Big\}\, ,\qquad \bar{\nu}\in\mathbb{N}_0^d\, .
\eeqq
For given $ f \, \in S_{p,q}^t F(\Omega)$ let $\ce f$ be an element of $ S_{p,q}^t F(\R)$ s.t. 
\[
\| \, \ce f \, | S_{p,q}^t F(\R)\| \le 2 \, \| \, f \, | S_{p,q}^t F(\Omega)\|
\qquad \mbox{and} \qquad (\ce f)_{|_\Omega} = f \, .
\]
We define
\[
g:= \sum_{\bar{\nu} \in \N_0^d} \sum_{\bar{m} \in A_{\bar{\nu}}^{\Omega}} 2^{|\bar{\nu}|_1} \, \langle \ce f, \Psi_{\bar{\nu},\bar{m}} \rangle \, \Psi_{\bar{\nu}, \bar{m}}\, .
\]
Then it follows that $g \in S_{p,q}^t F(\R)$, $g_{|_\Omega} = f $, 
\[
\supp g \subset \{x \in \R: ~ \max_{j=1, \ldots \, , d} |x_j| \le c_1\} \quad \mbox{and} \qquad 
\| \, g \, | S_{p,q}^t F(\R)\| \le c_2 \, \| \, f \, | S_{p,q}^t F(\Omega)\|\, .
\]
Here $c_1,c_2$ are independent of $f$. For this reason we define the following sequence spaces

\begin{definition}
Let $0<p<\infty$, $0< q\leq \infty$ and $t\in \mathbb{R}$.
\begin{enumerate}
\item If 
$$
\lambda=\lbrace \lambda_{\bar{\nu},\bar{m}}\in\mathbb{C}:\bar{\nu}\in \mathbb{N}_0^d,\ \bar{m}\in A_{\bar{\nu}}^{\Omega} \rbrace \, ,
$$
then we define
\[
s_{p,q}^{t,\Omega}f :=\Big\lbrace\lambda: \| \lambda|s_{p,q}^{t,\Omega}f\| =\Big\|\Big(\sum_{\bar{\nu}\in \mathbb{N}_0^d}\sum_{\bar{m}\in A_{\bar{\nu}}^{\Omega}}|2^{|\bar{\nu}|_1 t}\lambda_{\bar{\nu},\bar{m}}
\chi_{\bar{\nu},\bar{m}}(\cdot) |^q\Big)^{\frac{1}{q}}\Big|L_p(\mathbb{R}^d)\Big\|<\infty \Big\rbrace\, .
\]
\item If $\mu\in \N_0$ and 
$$
\lambda=\lbrace \lambda_{\bar{\nu},\bar{m}}\in\mathbb{C}:\bar{\nu}\in \mathbb{N}_0^d,\ |\bar{\nu}|_1=\mu,\ \bar{m}\in A_{\bar{\nu}}^{\Omega} \rbrace \, ,
$$
then we define
$$(s_{p,q}^{t,\Omega}f)_{\mu}=\Big\lbrace\lambda: \| \lambda|(s_{p,q}^{t,\Omega}f)_{\mu}\| =
\Big\|\Big(\sum_{|\bar{\nu}|_1=\mu}\sum_{\bar{m}\in A_{\bar{\nu}}^{\Omega}}|2^{|\bar{\nu}|_1t}\lambda_{\bar{\nu},\bar{m}}\chi_{\bar{\nu},\bar{m}}(\cdot) 
|^q\Big)^{\frac{1}{q}}\Big|L_p(\mathbb{R}^d)\Big\|<\infty \Big\rbrace \, .$$
\end{enumerate}

\end{definition}

Later on we shall need the following lemmas, see \cite{Hansen,KiSi,Vy}.

\begin{lemma}\label{ba1}
\begin{enumerate}
\item Let $\bar{\nu}\in \mathbb{N}_0^d$ and $\mu\in \mathbb{N}_0$. Then we have
$$ \#(A_{\bar{\nu}}^{\Omega})\asymp 2^{|\bar{\nu}|_1}\qquad\text{and}\qquad D_{\mu}=\sum_{|\bar{\nu}|_1=\mu}\#(A_{\bar{\nu}}^{\Omega})\asymp \mu^{d-1}2^{\mu}. \ $$
The equivalence constants do not depend on $\mu\in \mathbb{N}_0$.
\item Let $0<p< \infty$ and $t\in \mathbb{R}$. Then
$$(s_{p,p}^{t,\Omega}f)_{\mu}=2^{\mu(t-\frac{1}{p})}\ell_p^{D_{\mu}}\,,\qquad \mu\in \mathbb{N}_0\,.$$
\end{enumerate}
\end{lemma}
\begin{lemma}\label{ba2}
\begin{enumerate}
\item Let $0<p_1,p_2<\infty$, $0<q_1,q_2\leq \infty$ and $t\in \mathbb{R}$. Then
\[
\| \, id^*_{\mu} \, : (s^{t,\Omega}_{p_1,q_1}f)_\mu \to (s^{0,\Omega}_{p_2,q_2}f)_\mu\|
\lesssim 2^{\mu\big(-t+(\frac{1}{p_1}-\frac{1}{p_2})_+\big)}\mu^{(d-1)(\frac{1}{q_2} - \frac{1}{q_1})_+}
\]
with a constant behind $\lesssim$ independent of $\mu\in \mathbb{N}_0$.
\item Let $0<p_1<p_2<\infty$, $0<q_1,q_2\leq \infty$ and $t \in \re$. Then
$$ \|id^*_{\mu}: (s_{p_1,q_1}^{t,\Omega}f)_{\mu}\to (s_{p_2,q_2}^{0,\Omega}f)_{\mu} \| \lesssim 2^{\mu(-t+\frac{1}{p_1}-\frac{1}{p_2})}$$
with a constant behind $\lesssim$ independent of $\mu\in \mathbb{N}_0$.
\end{enumerate}
\end{lemma}

\section{Weyl and Bernstein numbers of embeddings of sequence spaces}\label{sec-seq}
\subsection{Some preparations}
We define the operators
 $$ id_{\mu}^* :\ s_{p_1,2}^{t,\Omega}f\to s_{p_2,2}^{0,\Omega}f $$
 where 
 $$
 (id_{\mu}^*\lambda )_{\bar{\nu},\bar{m}}=\begin{cases}
 \lambda_{\bar{\nu},\bar{m}}&\text{if}\ |\bar{\nu}|=\mu\,,\\
 0&\text{otherwise}\,.
 \end{cases}
 $$ 
 We split the identity $ id^*: s_{p_1,2}^{t,\Omega}f\to s_{p_2,2}^{0,\Omega}f $ into the sum of identities between building blocks
 $$id^*=\sum_{\mu=0}^{J}id_{\mu}^*+\sum_{\mu=J+1}^{L}id_{\mu}^*+\sum_{\mu=L+1}^{\infty}id_{\mu}^*,\ \ J<L\,.$$
 We will show how to choose $L$ and $J$ later. By the properties of Weyl numbers we have
 $$x_n(id^*)\leq \sum_{\mu=0}^{J}x_{n_{\mu}}(id_{\mu}^*)+\sum_{\mu=J+1}^{L}x_{n_{\mu}}(id_{\mu}^*)+\sum_{\mu=L+1}^{\infty} \|id_{\mu}^*\|$$
 where $n-1=\sum_{\mu=0}^{L}(n_{\mu}-1)$.
Lemma \ref{ba2} results in
 $$ \|id_{\mu}^*  \|\lesssim 2^{-\mu\big(t-(\frac{1}{p_1}-\frac{1}{p_2})_+\big)}, $$
 which implies 
 $$\sum_{\mu=L+1}^{\infty} \|id_{\mu}^*\|\lesssim  2^{-L \big(t-(\frac{1}{p_1}-\frac{1}{p_2})_+\big)}\,,$$
Next  we define $n_{\mu}$ as follows
 $$n_{\mu}=D_{\mu}+1,\ \ \mu=0,1,....,J\, .$$
 Then we get
 $$\sum_{\mu=0}^{J}n_{\mu}\asymp J^{d-1}2^J \qquad\text{and}\qquad  \sum_{\mu=0}^{J}x_{n_{\mu}}(id_{\mu}^*)=0.$$
 Summing up we have proved
 \begin{equation}\label{sum}
 x_n(id^*)\lesssim \sum_{\mu=J+1}^{L}x_{n_{\mu}}(id_{\mu}^*)+ 2^{-L \big(t-(\frac{1}{p_1}-\frac{1}{p_2})_+\big)} \, .
 \end{equation}
 To estimate the lower bound we use the following lemma. Recall, $\omega_n$ denotes either $x_n$ or $b_n$.
 
  \begin{lemma}\label{lowa}
  Let $1\leq p_1,p_2<\infty$. For all $\mu \in \N_0$ and $n \in \N$ we have 
  \beqq
  \omega_n\big(id_{\mu}^* :\ (s^{t,\Omega}_{p_1,2}f)_{\mu}\to  (s^{0,\Omega}_{p_2,2}f)_{\mu}\big)\leq \omega_n\big(id^* :\ s^{t,\Omega}_{p_1,2}f \to s^{0,\Omega}_{p_2,2}f \big) \, .
\eeqq
  \end{lemma}
  \begin{proof}
 The proof is carried out as in \cite[Lemma 6.10]{KiSi}. We consider the following diagram
 \beqq
 \begin{CD}
 s^{t,\Omega}_{p_1,2}f @ > id^* >> s^{0,\Omega}_{p_2,2}f\\
 @A id^1 AA @VV id^2 V\\
 (s^{t,\Omega}_{p_1,2}f)_\mu @ > id_{\mu}^* >> (s^{0,\Omega}_{p,2}f)_\mu \, .
 \end{CD}
 \eeqq
 Here $id^1$ is the canonical embedding and $id^2$ is the canonical projection.
 Since $id_{\mu}^* = id^2 \circ id^* \circ id^1$ the property (c) in the definition of $s-$numbers yields
 $$
 \omega_n(id_{\mu}^*) \leq \| \, id^1\, \| \cdot \|\, id^2 \, \| \cdot \omega_n(id^*) = \omega_n(id^*)\, .
 $$
 This completes the proof.
  \end{proof}
 \subsection{Weyl and Bernstein numbers of embeddings $id_{\mu}^* :\ (s^{t,\Omega}_{p_1,2}f)_{\mu}\to  (s^{0,\Omega}_{p_2,2}f)_{\mu}$}
 The following lemma holds for both, Weyl and Bernstein numbers, since they share property (c) in the definition of $s-$numbers.
 \begin{lemma}\label{ba4}Let $t\in \re$ and $1\leq p_1,p_2<\infty$.
 \begin{enumerate}
 \item If $1\leq p_1 \leq 2$, then we have
 \be\label{ba4-1}\mu^{-(d-1)(\frac{1}{p_2}-\frac{1}{2})_+}2^{\mu( -t+\frac{1}{2} -\frac{1}{p_2})}\omega_n (id_{2,p_2}^{D_{\mu}})  \lesssim \omega_n(id_{\mu}^*).
 \ee
 \item  If $\epsilon>0$ such that $p_1-\epsilon>0$, then
 \be\label{ba4-2}
\mu^{-(d-1)(\frac{1}{p_2}-\frac{1}{2})_+} 2^{\mu(-t+\frac{1}{p_1}-\frac{1}{p_2})}\omega_n(id_{p_1-\epsilon,p_2}^{D_{\mu}})\lesssim \omega_n(id_{\mu}^*).
 \ee
 \item If  $2\leq p_2<\infty$, then
 \be\label{ba4-3}\omega_n(id_{\mu}^*)\lesssim \mu^{(d-1)(\frac{1}{p_1}-\frac{1}{2})_+}2^{\mu(-t+\frac{1}{p_1}-\frac{1}{p_2})}\omega_n(id_{p_1,2}^{D_{\mu}}).
 \ee
 \end{enumerate}
 \end{lemma}
 \begin{proof}
 {\it Step 1.} Proof of (i). We consider the following diagram
 \beqq
 \begin{CD}
 (s^{t,\Omega}_{p_1,2}f)_\mu  @ > id_{\mu}^* >> (s^{0,\Omega}_{p_2,2}f)_\mu \\
 @A id_1 AA @VV id_3 V\\
 (s^{t,\Omega}_{2,2}f)_\mu @ > id_{2} >> (s^{0,\Omega}_{p_2,p_2}f)_\mu \, .
 \end{CD}
 \eeqq
and obtain
\be\label{ba4-01}
\omega_n(id_{2} )\ \lesssim\ \| id_1 \|\cdot \|id_3\|\cdot\omega_n(id_{\mu}^*).
\ee
 By Lemma \ref{ba2} (i), Lemma \ref{ba1} (ii) we have
 $$\|  id_1\|\lesssim  1\,,\qquad \|  id_3\|\lesssim  \mu^{(d-1)(\frac{1}{p_2}-\frac{1}{2})_+}$$
 and
 $$\omega_n(id_2)\asymp 2^{\mu(-t+\frac{1}{2}-\frac{1}{p_2})}\omega_n(id_{2,p_2}^{D_{\mu}})\,.$$
 This together with \eqref{ba4-01} results in \eqref{ba4-1}.\\
 {\it Step 2.} Proof of (ii). We consider the following diagram
 \beqq
 \begin{CD}
 (s^{t,\Omega}_{p_1,2}f)_\mu  @ > id_{\mu}^* >> (s^{0,\Omega}_{p_2,2}f)_\mu \\
 @A id_1 AA @VV id_3 V\\
 (s^{0,\Omega}_{p_1-\epsilon,p_1-\epsilon}f)_\mu @ > id_{2} >> (s^{0,\Omega}_{p_2,p_2}f)_\mu \, .
 \end{CD}
 \eeqq
Property (c) yields
 $$\omega_n(id_{2} )\ \leq\ \| id_1 \|\cdot\|id_3\|\cdot \omega_n(id_{\mu}^*).$$
This together with
 $$ \|  id_1\|\lesssim 2^{\mu(t+\frac{1}{p_1-\epsilon}-\frac{1}{p_1})}\, ,\qquad \|  id_3\|\lesssim \mu^{(d-1)(\frac{1}{p_2}-\frac{1}{2})_+},$$
see Lemma \ref{ba2}, and
 $$\omega_n(id_2)\asymp 2^{\mu(-\frac{1}{p_2}+\frac{1}{p_1-\epsilon})}\omega_n(id_{p_1-\epsilon,p_2}^{D_{\mu}}),$$
 see Lemma \ref{ba1}, claims the estimate.\\
 {\it Step 3.} Proof of (iii). This time we consider the following diagram
 \beqq
 \begin{CD}
 (s^{t,\Omega}_{p_1,p_1}f)_\mu  @ > id_2 >> (s^{0,\Omega}_{2,2}f)_\mu \\
 @A id_1 AA @VV id_3 V\\
 (s^{t,\Omega}_{p_1,2}f)_\mu @ > id_{\mu}^* >> (s^{0,\Omega}_{p_2,2}f)_\mu \, 
 \end{CD}
 \eeqq
and obtain
\be\label{ba41}
\omega_n(id_{\mu}^* )\ \leq\ \| id_1 \|\cdot\|id_3\|\cdot \omega_n(id_{2})\,.
\ee
Employing Lemmas \ref{ba2} and \ref{ba1} we have
 $$\|  id_1\|\lesssim \mu^{(d-1)(\frac{1}{p_1}-\frac{1}{2})_+} \, ,\qquad\   \|id_3\|\lesssim 2^{\mu(\frac{1}{2}-\frac{1}{p_2})}$$
 and
 $$\omega_n(id_2)\asymp 2^{\mu(-t+\frac{1}{p_1}-\frac{1}{2})}\omega_n(id_{p_1,2}^{D_{\mu}}).$$
 From this and \eqref{ba41} the claim follows.
 The proof is complete.
 \end{proof}
 \begin{lemma}\label{low}
 Let $1<p_2\leq 2<p_1<\infty$. Then we have
 \beqq
 \omega_n (id_{\mu}^*: (s^{t,\Omega}_{p_1,2}f)_{\mu}\to (s^{0,\Omega}_{p_2,2}f)_{\mu})&\gtrsim& 2^{-t\mu},\ \ n=\Big[\mu^{d-1}2^{\frac{2\mu}{p_1}}\Big].
 \eeqq
 \end{lemma}
 \begin{proof}
 {\it Step 1.} We concentrate on Bernstein numbers. It will be convenient for us to introduce the subspaces $(s_{p,q}^{t,\Omega}b)_{\mu}$. If $0<p, q\leq \infty$, $t\in \mathbb{R}$, $\mu\in \mathbb{N}_0$ and 
 $$\lambda=\lbrace \lambda_{\bar{\nu},\bar{m}}\in\mathbb{C}:\bar{\nu}\in \mathbb{N}_0^d,\ |\bar{\nu}|_1=\mu\,,\ \bar{m}\in A_{\bar{\nu}}^{\Omega} \rbrace \, ,$$
 then we define
 $$(s_{p,q}^{t,\Omega}b)_{\mu}=\Big\lbrace\lambda: \| \lambda|(s_{p,q}^{t,\Omega}b)_{\mu}\| =
 \Big(\sum_{|\bar{\nu}|_1=\mu}2^{|\bar{\nu}|_1(t-\frac{1}{p})q}\big(\sum_{\bar{m}\in A_{\bar{\nu}}^{\Omega}}|\lambda_{\bar{\nu},\bar{m}} |^p\big)^{\frac{q}{p}}\Big)^{\frac{1}{q}}<\infty \Big\rbrace\, .$$
 These subspaces are discussed in Vybiral \cite[Chapter 3]{Vy} and Hansen \cite[Chapter 5]{Hansen}. Since $p_2\leq 2<p_1 $ we have the chain of embeddings
 \be\label{chain}
 (s^{t,\Omega}_{p_1,2}b)_\mu \overset{id^1}{ \longrightarrow } (s^{t,\Omega}_{p_1,2}f)_\mu \overset{id_{\mu}^*}{\longrightarrow}(s^{0,\Omega}_{p_2,2}f)_\mu \overset{id^2}{\longrightarrow}(s^{0,\Omega}_{p_2,2}b)_\mu\, ,
 \ee
 with the norms of $id^1$ and $id^2$ independent of $\mu$ 
 see \cite[Lemma 5.3.4]{Hansen}. From the definition of Bernstein numbers and \eqref{chain} we deduce the existence of some constant $C>0$ such that
 \beq\label{low0}
  b_n(id_{\mu}^*: (s^{t,\Omega}_{p_1,2}f)_{\mu}\to (s^{0,\Omega}_{p_2,2}f)_{\mu})
  &=&\sup_{L_n}\inf_{\lambda\in L_n}\frac{\|\lambda\,|\,(s^{0,\Omega}_{p_2,2}f)_{\mu} \|}{\|\lambda\,|\,(s^{t,\Omega}_{p_1,2}f)_{\mu} \| }\nonumber\\
  &\geq &C\sup_{L_n}\inf_{\lambda\in L_n}\frac{\|\lambda\,|\,(s^{0,\Omega}_{p_2,2}b)_{\mu} \|}{\|\lambda\,|\,(s^{t,\Omega}_{p_1,2}b)_{\mu} \| }\, ,
   \eeq
 where $C$ is independent of $n$. Recall that the supremum is taken over all linear subspace $L_n$ of dimension $n= \Big[\mu^{d-1}2^{\frac{2\mu}{p_1}}\Big]$ in $(s_{2,2}^{0,\Omega}f)_{\mu}$. Note that 
   \be\label{low1}
   \frac{\|\lambda\, |\,(s^{0,\Omega}_{p_2,2}b)_{\mu} \|}{\|\lambda\,|\,(s^{t,\Omega}_{p_1,2}b)_{\mu} \| }  \ = \ \frac{2^{\mu(-t+\frac{1}{p_1} -\frac{1}{p_2})} \Big(\sum\limits_{|\bar{\nu}|_1=\mu}\Big(\sum\limits_{\bar{m}\in A_{\bar{\nu}}^{\Omega}}|\lambda_{\bar{\nu},\bar{m}} |      ^{p_2}\Big)^{\frac{2}{p_2}}\Big)^{\frac{1}{2}}}{ \Big(\sum\limits_{|\bar{\nu}|_1=\mu}\Big(\sum\limits_{\bar{m}\in A_{\bar{\nu}}^{\Omega}}|\lambda_{\bar{\nu},\bar{m}} |^{p_1}\Big)^{\frac{2}{p_1}}\Big)^{\frac{1}{2}} }.
   \ee
 We put $\Delta_{\mu}=\{\bar{\nu}\in \N_0^d\ :\ |\bar{\nu}|_1=\mu\}$. For each $\bar{\nu}\in \Delta_{\mu}$ the inequality 
   \be \label{low1-1}
    b_k(id_{p_1,p_2}^{|A_{\bar{\nu}}^{\Omega}|})\gtrsim 2^{|\bar{\nu}|_1(\frac{1}{p_2}-\frac{1}{p_1})},\ \  k= \big[2^{|\bar{\nu}|_1\frac{2}{p_1}}\big],
    \ee 
see Lemma \ref{Bern1} (ii), implies that  there exists a linear subspace $L_k^{\bar{\nu}}$ in $\re^{|A_{\bar{\nu}}^{\Omega}|}\times \re^{|\Delta_{\mu}|}$ of dimension $ k= \big[2^{|\bar{\nu}|_1\frac{2}{p_1}}\big]$ such that
 \beqq
 \inf_{\lambda \in L_k^{\bar{\nu}} }\frac{\Big(\sum\limits_{\bar{m}\in A_{\bar{\nu}}^{\Omega}}|\lambda_{\bar{\nu},\bar{m}} |^{p_2}\Big)^{\frac{1}{p_2}}}{\Big(\sum\limits_{\bar{m}\in A_{\bar{\nu}}^{\Omega}}|\lambda_{\bar{\nu},\bar{m}} |^{p_1}\Big)^{\frac{1}{p_1}}}&\gtrsim& \frac{2^{|\bar{\nu}|_1(\frac{1}{p_2}-\frac{1}{p_1})}}{2} \, .
 \eeqq
Here the constant behind $\gtrsim$ is the same as in \eqref{low1-1}. Consequently
 \be\label{low2}
 \Big(\sum\limits_{\bar{m}\in A_{\bar{\nu}}^{\Omega}}|\lambda_{\bar{\nu},\bar{m}} |^{p_1}\Big)^{\frac{1}{p_1}}\ \lesssim \  2^{-|\bar{\nu}|_1(\frac{1}{p_2}-\frac{1}{p_1})}  \Big(\sum\limits_{\bar{m}\in A_{\bar{\nu}}^{\Omega}}|\lambda_{\bar{\nu},\bar{m}} |^{p_2}\Big)^{\frac{1}{p_2}}.
 \ee
 holds for all $\lambda \in L_k^{\bar{\nu}}$. We put 
  $$L^{\mu}=\bigoplus_{|\bar{\nu}|_1=\mu}L_{k}^{\bar{\nu}}\,.$$
 Obviously $\text{dim}\, L^{\mu}\asymp \big[\mu^{d-1}2^{\mu\frac{2}{p_1}}\big] $. Inserting \eqref{low2} into \eqref{low1} we have found
 \beqq
  \frac{\|\lambda\ |(s^{0,\Omega}_{p_2,2}b)_{\mu} \|}{\|\lambda|(s^{t,\Omega}_{p_1,2}b)_{\mu} \| } 
  \ \gtrsim\ \frac{2^{\mu(-t+\frac{1}{p_1} -\frac{1}{p_2})} \Big(\sum\limits_{|\bar{\nu}|_1=\mu}\Big(\sum\limits_{\bar{m}\in A_{\bar{\nu}}^{\Omega}}|\lambda_{\bar{\nu},\bar{m}} |^{p_2}\Big)^{\frac{2}{p_2}}\Big)^{\frac{1}{2}}}{ \Big(\sum\limits_{|\bar{\nu}|_1=\mu}2^{-2|\bar{\nu}|_1(\frac{1}{p_2}-\frac{1}{p_1})} \Big(\sum\limits_{\bar{m}\in A_{\bar{\nu}}^{\Omega}}|\lambda_{\bar{\nu},\bar{m}} |^{p_2}\Big)^{\frac{2}{p_2}}\Big)^{\frac{1}{2}} }\ =\ 2^{-t\mu}
  \eeqq
 for all $\lambda\in L^{\mu}$. 
   In a view of \eqref{low0} the desired result follows.\\
  {\it Step 2.} We prove that
  \be\label{low2.1}
  b_n\big(id:\ (s_{p_1,2}^{t,\Omega}f)_{\mu}\to (s_{2,2}^{0,\Omega}f)_{\mu}\big) \leq x_n\big(id:\ (s_{p_1,2}^{t,\Omega}f)_{\mu}\to (s_{2,2}^{0,\Omega}f)_{\mu}\big)\, 
  \ee
 for $1<p_1<\infty$.  By $p_1'$ we denote the conjugate number of $p_1$. From Lemma \ref{bern-gel} and the duality of Kolmogorov and Gelfand numbers, see \cite[11.7.7]{Pi80-2}, we deduce
  \beq\label{low3}
  b_n\big(id:\ (s_{p_1,2}^{t,\Omega}f)_{\mu}\to (s_{2,2}^{0,\Omega}f)_{\mu}\big)&=
  &\big[c_{D_{\mu}-n+1}\big(id:\ (s_{2,2}^{0,\Omega}f)_{\mu}\to (s_{p_1,2}^{t,\Omega}f)_{\mu}\big)\big]^{-1}\nonumber \\
  &=& \big[d_{D_{\mu}-n+1}\big(id:\ (s_{p'_1,2}^{-t,\Omega}f)_{\mu}\to (s_{2,2}^{0,\Omega}f)_{\mu}\big) \big]^{-1}\, .
  \eeq
 Let $L_{D_{\mu}-n}$ be a subspace of $(s_{2,2}^{0,\Omega}f)_{\mu}$ with orthonormal basis $O^*=\{e_i^*,\ i=1,...,(D_{\mu}-n)\}$. By $O=\{e^j,\ j=1,...,n\}$ we denote an orthonormal system in $(s_{2,2}^{0,\Omega}f)_{\mu}$ such that $\{O^*,O\}$ is an orthonormal basis of $(s_{2,2}^{0,\Omega}f)_{\mu}$. Denote $(s_{2,2}^{0,\Omega}f)_{\mu,n}$  the span of $O$ with the norm induced from $(s_{2,2}^{0,\Omega}f)_{\mu}$. From the definition of Kolmogorov numbers, see \eqref{def1}, we have
  \beqq
  d_{D_{\mu}-n+1}\big(id:\ (s_{p'_1,2}^{-t,\Omega}f)_{\mu}\to (s_{2,2}^{0,\Omega}f)_{\mu}\big) &=&\inf_{L_{D_{\mu}-n}}\sup_{\|\lambda|(s_{p'_1,2}^{-t,\Omega}f)_{\mu}\|=1 }\inf_{\lambda_1\in L_{D_{\mu}-n}}\|\lambda-\lambda_1|(s_{2,2}^{0,\Omega}f)_{\mu}\|\\
  &=&\inf_{L_{D_{\mu}-n}}\sup_{\|\lambda|(s_{p'_1,2}^{-t,\Omega}f)_{\mu}\|=1 }\Big\| \sum_{j=1}^n \langle \lambda,e^j \rangle e^j\Big|(s_{2,2}^{0,\Omega}f)_{\mu}\Big\|\\
   &=&\inf_{O}\sup_{\|\lambda|(s_{p'_1,2}^{-t,\Omega}f)_{\mu}\|=1 }\Big\| \sum_{j=1}^n \langle \lambda,e^j \rangle e^j\Big|(s_{2,2}^{0,\Omega}f)_{\mu}\Big\|.
  \eeqq
  The infimum is taken over all orthonormal systems $O=\{e^j,\ j=1,...,n\}$. If we denote by $\text{Pr}$ the projection from $(s_{p'_1,2}^{-t,\Omega}f)_{\mu}$ onto $(s_{2,2}^{0,\Omega}f)_{\mu,n}$, then we get
 \be\label{low4}
 d_{D_{\mu}-n+1}\big(id:\ (s_{p'_1,2}^{-t,\Omega}f)_{\mu}\to (s_{2,2}^{0,\Omega}f)_{\mu}\big) =\inf_{O} \| \text{Pr}:\ (s_{p'_1,2}^{-t,\Omega}f)_{\mu}\to (s_{2,2}^{0,\Omega}f)_{\mu,n}\|.
  \ee
 Property (c) yields
  \beqq
  x_n\big(J: (s_{2,2}^{0,\Omega}f)_{\mu,n}
  &\to& (s_{2,2}^{0,\Omega}f)_{\mu}\big)\\
  &\leq & \| J_1:\ (s_{2,2}^{0,\Omega}f)_{\mu,n}\to (s_{p_1,2}^{t,\Omega}f)_{\mu}\|\cdot x_n\big(id:\ (s_{p_1,2}^{t,\Omega}f)_{\mu}\to (s_{2,2}^{0,\Omega}f)_{\mu}\big)\,.
  \eeqq
 Here $J$ and $J_1$ are injections from respective spaces. Note that $\text{Pr}$ is the adjoint operator of $J_1$. Hence we have
  \beqq
   x_n\big(J: (s_{2,2}^{0,\Omega}f)_{\mu,n}
   &\to& (s_{2,2}^{0,\Omega}f)_{\mu}\big)\\
   &\leq &  \| \text{Pr}:\ (s_{p'_1,2}^{-t,\Omega}f)_{\mu}\to (s_{2,2}^{0,\Omega}f)_{\mu,n}\|\cdot x_n\big(id:\ (s_{p_1,2}^{t,\Omega}f)_{\mu}\to (s_{2,2}^{0,\Omega}f)_{\mu}\big).
  \eeqq
 The equality 
  \beqq
   x_n\big(J: (s_{2,2}^{0,\Omega}f)_{\mu,n}
   \to (s_{2,2}^{0,\Omega}f)_{\mu}\big)= a_n\big(J: (s_{2,2}^{0,\Omega}f)_{\mu,n}
      \to (s_{2,2}^{0,\Omega}f)_{\mu}\big)=1\,,
  \eeqq
 see \eqref{an-xn}, implies
  \beqq
1\ \leq\  \| \text{Pr}:\ (s_{p'_1,2}^{-t,\Omega}f)_{\mu}\to (s_{2,2}^{0,\Omega}f)_{\mu,n}\|\cdot  x_n\big(id:\ (s_{p_1,2}^{t,\Omega}f)_{\mu}\to (s_{2,2}^{0,\Omega}f)_{\mu}\big)\,.
  \eeqq
  This, in connection with \eqref{low4}, results in 
  \beqq
\big[d_{D_{\mu}-n+1}\big(id:\ (s_{p'_1,2}^{-t,\Omega}f)_{\mu}\to (s_{2,2}^{0,\Omega}f)_{\mu}\big)\big]^{-1}\ \leq \  x_n\big(id:\ (s_{p_1,2}^{t,\Omega}f)_{\mu}\to (s_{2,2}^{0,\Omega}f)_{\mu}\big)\,.
  \eeqq  
In view of \eqref{low3} the inequality \eqref{low2.1} follows.\\
   {\it Step 3.} Let $p_2<2<p_1$. There exists some $\theta\in (0,1)$ such that $\frac{1}{2}=\frac{1-\theta}{p_2}+\frac{\theta}{p_1}$ and consequently
 \beqq
 \| \lambda| (s^{0,\Omega}_{2,2}f)_{\mu}\|\leq \|\lambda| (s^{0,\Omega}_{p_2,2}f)_{\mu}\|^{\theta}\cdot \|\lambda| (s^{0,\Omega}_{p_1,2}f)_{\mu}\|^{1-\theta}
 \eeqq  
for all $\lambda\in (s^{0,\Omega}_{2,2}f)_{\mu}$. Now the interpolation property of the Weyl numbers, see Proposition. \ref{inter}, and property (a) of $s$-number yield 
   \beqq
  x_n\big(id:\ (s_{p_1,2}^{t,\Omega}f)_{\mu}&\to& (s_{2,2}^{0,\Omega}f)_{\mu}\big)\\
 &\leq& x_n^{1-\theta}\big(id:\ (s_{p_1,2}^{t,\Omega}f)_{\mu}\to (s_{p_2,2}^{0,\Omega}f)_{\mu}\big)\cdot\| id:\ (s_{p_1,2}^{t,\Omega}f)_{\mu}\to (s_{p_1,2}^{0,\Omega}f)_{\mu} \|^{\theta}\\
 &\leq& x_n^{1-\theta}\big(id:\ (s_{p_1,2}^{t,\Omega}f)_{\mu}\to (s_{p_2,2}^{0,\Omega}f)_{\mu}\big)\cdot 2^{-t\mu\theta}
   \eeqq
  for $n\in \N$. Finally, choosing $n=\Big[\mu^{d-1}2^{\frac{2\mu}{p_1}}\Big]$ and taking into account  \eqref{low2.1} and Step 1 the claim follows for Weyl numbers as well. The proof is complete.
  \end{proof}
  \begin{remark}\rm
 {\rm (i)} The proof in Step 2 is similar to the proof of Satz 3.1 in \cite{Claus}.\\
 {\rm (ii)}  Lemma \ref{low} can be extended to the case $1\leq p_2\leq 2<p_1<\infty$ for Weyl numbers, see Step 3. That is, if $1\leq p_2\leq 2<p_1<\infty$, then we have
 \beqq
x_n (id_{\mu}^*: (s^{t,\Omega}_{p_1,2}f)_{\mu}\to (s^{0,\Omega}_{p_2,2}f)_{\mu})&\gtrsim& 2^{-t\mu},\ \ n=\Big[\mu^{d-1}2^{\frac{2\mu}{p_1}}\Big].
 \eeqq
  \end{remark}
There is an interesting relation of Weyl numbers and  absolutely $(r,s)$-summing norms. Let $1\leq s\leq r< \infty$.   An operator $T\in \mathcal{L}(X,Y)$ is said to be \emph{absolutely $(r,s)$-summing} if there is a constant 
  $C>0$ such that for 
  all $n\in \mathbb{N}$ and $x_1 , \dots, x_n \in X$ the inequality
  \begin{equation}\label{F2}
  \big (\sum_{j=1} ^n \|\, T x_j \, |Y\|^r \Big )^{1/r} \leq C 
  \sup_{x^* \in X^*, \|x^*|X^*\|\leq 1} \Big ( \sum_{j=1} ^n |\langle x_j , x^* \rangle|^s\Big)^{1/s}
  \end{equation}
  holds (see \cite[Chapter 17]{Pi80-2} or \cite[Section 1.2]{Pi-87}). 
  The norm $\pi_{r,s} (T)$ is given by the infimum with respect to $C>0$ satisfying (\ref{F2}).
  $X^*$ refers to the dual space of $X$. In the literature sometimes the notions $\mathcal{B}_{r,s}(T)$ and $P_{r,s}(T)$ are used instead of $\pi_{r,s}(T)$. If $r=s$ we write $\pi_{r}(T)$ instead of $\pi_{r,s}(T)$. The announced relation between Weyl numbers and
   the $(r,s)$-summing norms is given by the following lemma, see \cite{Pi80-1}.
  \begin{lemma}\label{rs1}
 Let $X$, $Y$ be Banach spaces. Let $2\leq r<\infty$ and $T\in \pi_{r,2}(X,Y)$. Then for any $n\in \N$ we have
 \beqq
 x_n(T)\leq n^{-\frac{1}{r}}\pi_{r,2}(T).
 \eeqq
  \end{lemma}
  This will be used to prove the following proposition.
 \begin{proposition}\label{rs2}Let $t\in \re$, $2\leq  p_2<p_1< \infty$ and $
 \frac{1}{r}=\frac{1/p_2-1/p_1}{1-2/p_1}$. Then we have
 \beqq
 \pi_{r,2}\big(id_{\mu}^*: (s^{t,\Omega}_{p_1,2}f)_{\mu}\to (s^{0,\Omega}_{p_2,2}f)_{\mu}\big)\ \leq\ 2^{\mu(-t+\frac{1}{p_1}-\frac{1}{p_2})}D_{\mu}^{\frac{1}{r}}.
 \eeqq
 \end{proposition}
 \begin{proof} We consider the case   $2<p_2<p_1<\infty$. Let $\theta=\frac{1/p_2-1/p_1}{1/2-1/p_1}$. Then we find
 \beqq
 \frac{1}{p_2}=\frac{\theta}{2}+\frac{1-\theta}{p_1}\qquad \text{and}\qquad \frac{1}{r}=\frac{\theta}{2}+\frac{1-\theta}{\infty}.
 \eeqq
 By H\"older's inequality we obtain
 \beqq
 \| \lambda| (s^{0,\Omega}_{p_2,2}f)_{\mu}\|\leq \|\lambda| (s^{0,\Omega}_{2,2}f)_{\mu}\|^{\theta}\cdot \|\lambda| (s^{0,\Omega}_{p_1,2}f)_{\mu}\|^{1-\theta}
 \eeqq
 for all $\lambda\in (s^{0,\Omega}_{p_2,2}f)_{\mu}$. The definition of the absolutely $(r,s)$-summing norms yields that
 \beqq
 \pi_{r,2}\big(id_{\mu}^* :\ (s^{t,\Omega}_{p_1,2}f)_{\mu}&\to& (s^{0,\Omega}_{p_2,2}f)_{\mu}\big)\\
 &\leq & \pi^{\theta}_{2}\big(id:\ (s^{t,\Omega}_{p_1,2}f)_{\mu}\to (s^{0,\Omega}_{2,2}f)_{\mu}\big)\cdot \|id:\ (s^{t,\Omega}_{p_1,2}f)_{\mu}\to (s^{0,\Omega}_{p_1,2}f)_{\mu}\|^{1-\theta}\,.
 \eeqq
 Note that the chain of embeddings
 \beqq
 (s^{t,\Omega}_{p_1,2}f)_{\mu}\hookrightarrow (s^{t,\Omega}_{p_1,p_1}f)_{\mu}\hookrightarrow (s^{0,\Omega}_{2,2}f)_{\mu}
 \eeqq
  implies
  \be \label{pi}
  \pi_{2}\big(id:\ (s^{t,\Omega}_{p_1,2}f)_{\mu}\to (s^{0,\Omega}_{2,2}f)_{\mu}\big)\leq \pi_{2}\big(id:\ (s^{t,\Omega}_{p_1,p_1}f)_{\mu}\to (s^{0,\Omega}_{2,2}f)_{\mu}\big),
  \ee 
since $ \pi_{r,s}$ is an operator ideal, see \cite[Theorem 1.2.3]{Pi-87}. From this and Lemmas \ref{ba1}, \ref{ba2} we derive
 \beqq
 \pi_{r,2}\big(id_{\mu}^* :\ (s^{t,\Omega}_{p_1,2}f)_{\mu}\to (s^{0,\Omega}_{p_2,2}f)_{\mu}\big)&\lesssim& \pi^{\theta}_{2}\big(id:\ (s^{t,\Omega}_{p_1,p_1}f)_{\mu}\to (s^{0,\Omega}_{2,2}f)_{\mu}\big)\cdot 2^{-t\mu(1-\theta)}\\
 &\lesssim & \big[2^{\mu(-t+\frac{1}{p_1}-\frac{1}{2})}\pi_{2}\big(id:\ \ell_{p_1}^{D_{\mu}}\to \ell_{2}^{D_{\mu}})\big]^{\theta}\cdot 2^{-t\mu(1-\theta)}\, .
 \eeqq
 Finally, the equality $\pi_{2}\big(id:\ \ell_{p_1}^{m}\to \ell_{2}^{m})= m^{\frac{1}{2}}$, see  \cite[page 309]{Pi80-2}, yields the claimed estimate. The case $p_2=2$  is a consequence of \eqref{pi}. This finishes the proof.
 \end{proof}
 The following corollary is a consequence of Lemma \ref{rs1} and Proposition \ref{rs2}.
 \begin{corollary}\label{cor-weyl}Let $2\leq  p_2<p_1<\infty$. Then 
 \beqq
 x_n\big(id_{\mu}^*: (s^{t,\Omega}_{p_1,2}f)_{\mu}\to (s^{0,\Omega}_{p_2,2}f)_{\mu}\big)\lesssim 2^{\mu(-t+\frac{1}{p_1}-\frac{1}{p_2})}\Big(\frac{D_{\mu}}{n}\Big)^{\frac{1}{r}}\, ,\ \ \ \ \frac{1}{r}=\frac{1/p_2-1/p_1}{1-2/p_1}
 \eeqq
 holds for all $n\in \N$.
 \end{corollary}

\subsection{The results for Weyl numbers}
\begin{theorem}\label{seq-1}
 Let $1\leq p_1\leq 2\leq p_2< \infty$ and $t>\frac{1}{p_1}-\frac{1}{p_2}$. Then
\beqq 
x_n(id^*)\asymp  n^{-t+\frac{1}{2}-\frac{1}{p_2}}(\log n)^{(d-1)(t-\frac{1}{2}+\frac{1}{p_2})}\, ,\ \ \ \ n\geq 2.
\eeqq
\end{theorem}

\begin{proof}{\it Step 1.} Estimate from below. Because of $p_1\leq 2\leq p_2$, Lemma \ref{lowa} and \eqref{ba4-1} imply
$$x_n(id^*)\gtrsim 2^{\mu(-t+\frac{1}{2}-\frac{1}{p_2})}x_n(id_{2,p_2}^{D_{\mu}}).$$
We choose $n=\big[\dfrac{D_{\mu}}{2}\big]$. Then \eqref{wth1} yields $x_n(id_{2,p_2}^{D_{\mu}})\asymp1$. Hence
$$x_n(id^*)\gtrsim 2^{\mu(-t+\frac{1}{p_2}-\frac{1}{2})}. $$
Because of $2^{\mu}\asymp \frac{n}{(\log n)^{d-1}}$ we conclude 
\beqq
x_n(id^*)\gtrsim n^{-t+\frac{1}{2}-\frac{1}{p_2}}(\log n)^{(d-1)(t-\frac{1}{2}+\frac{1}{p_2})}.
\eeqq
{\it Step 2.} Estimate from above. We shall use \eqref{sum}. We choose $L> J$ such that
$$2^{L(-t+\frac{1}{p_1}-\frac{1}{p_2})} \lesssim 2^{J(-t+\frac{1}{2}-\frac{1}{p_2})}\,.$$
Then we obtain from \eqref{sum}
\beqq
x_n(id^*)\lesssim \sum_{\mu=J+1}^{L}x_{n_{\mu}}(id_{\mu}^*)+  2^{J(-t+\frac{1}{2}-\frac{1}{p_2})}.
\eeqq
We define 
$$ n_{\mu}= D_{\mu}2^{(J-\mu)\lambda}\leq \dfrac{D_{\mu}}{2},\ \ J+1\leq \mu\leq L$$ 
with $\lambda$ satisfying the relations
\begin{equation}\label{cond2}
\lambda>1\qquad\text{and}\qquad t+\frac{1}{p_2}-\frac{1}{2}>\lambda\bigg(\dfrac{1}{p_1}-\dfrac{1}{2}\bigg).
\end{equation}
This implies
$$\sum_{\mu=J+1}^{L}n_{\mu}\asymp J^{d-1}2^J.$$
Now \eqref{ba4-3} and \eqref{wth2} yield
\beqq
x_{n_{\mu}}(id_{\mu}^*)
&\lesssim & \mu^{(d-1)(\frac{1}{p_1}-\frac{1}{2})}2^{\mu(-t+\frac{1}{p_1}-\frac{1}{p_2})}x_{n_\mu}(id_{p_1,2}^{D_{\mu}})\\
&\lesssim& \mu^{(d-1)(\frac{1}{p_1}-\frac{1}{2})}2^{\mu(-t+\frac{1}{p_1}-\frac{1}{p_2})}(D_{\mu}2^{(J-\mu)\lambda})^{\frac{1}{2}-\frac{1}{p_1}}\\
&\asymp& 2^{\mu(-t+\frac{1}{2}-\frac{1}{p_2})}  2^{(J-\mu)\lambda(\frac{1}{2}-\frac{1}{p_1})}\,.
\eeqq
Taking into account  the condition (\ref{cond2}), we obtain
$$\sum_{\mu=J+1}^{L}x_{n_{\mu}}(id_{\mu}^*)\lesssim 2^{J(-t+\frac{1}{2}-\frac{1}{p_2})}. $$
Consequently we get
$$x_n(id^*)\lesssim 2^{J(-t+\frac{1}{2}-\frac{1}{p_2})}.$$
Notice that  $n=n_J=c.2^JJ^{(d-1)}$. Without loss of generality we assume that
\[
A\, J^{d-1}\, 2^J\le n_J \le B \, J^{d-1}\, 2^J\, , \qquad J \in \N\, ,
\] 
for some $A,B\in \N$ independent of $n$. 
 Then we conclude from the monotonicity of the Weyl numbers
\[
x_{B \, J^{d-1}\, 2^J}(id^*) \lesssim 
\Big(\frac{B \, J^{d-1}\, 2^J}{\log^{d-1} (B \, J^{d-1}\, 2^J)}\Big)^{- t+\frac{1}{2}-\frac{1}{p_2}} \, .
\]
Employing one more times the monotonicity of the Weyl numbers and in addition its polynomial behaviour we can switch 
from the subsequence $(B \, J^{d-1}\, 2^J)_J$ to $n\in \N$ in this formula by possibly changing the constant behind $\lesssim$.
This finishes our proof.
\end{proof}

\begin{theorem}\label{seq-2}
Let $1\leq p_2, p_1\leq 2$ and $t>\Big(\frac{1}{p_1}-\frac{1}{p_2}\Big)_+$. Then we have
$$x_n(id^*)\asymp n^{-t}(\log n)^{(d-1)t}\, ,\ \ \ \ \  n\geq 2.$$
\end{theorem}
\begin{proof}
{\it Step 1.} Estimate from below. Since $p_1,p_2\leq 2$, from Lemma \ref{lowa} and \eqref{ba4-1} we have
$$x_n(id^*) \gtrsim \mu^{(d-1)(\frac{1}{2}-\frac{1}{p_2})}2^{\mu(-t+\frac{1}{2}-\frac{1}{p_2})}x_n(id_{2,p_2}^{D_{\mu}} ). $$
By choosing $n=\big[\dfrac{D_{\mu}}{2}\big]$ together with \eqref{wth4} we obtain
\beqq
x_n(id^*) 
&\gtrsim & \mu^{(d-1)(\frac{1}{2}-\frac{1}{p_2})}2^{\mu(-t+\frac{1}{2}-\frac{1}{p_2})}(D_{\mu})^{\frac{1}{p_2}-\frac{1}{2}}\\
&\asymp & 2^{\mu(-t)}.
\eeqq
Because of $2^{\mu}\asymp \frac{n}{(\log n)^{d-1}}$ the desired estimate follows.\\
{\it Step 2.} Estimate from above in case $1\leq p_2\leq p_1\leq 2$ and $t>0$. For $J\in \mathbb{N}$ and $\lambda\in s^{t,\Omega}_{p_1,2}f$ we put
\[
S_J\lambda :=\sum_{\mu=0}^{J}\sum_{|\bar{\nu}|_1=\mu}\sum_{\bar{m}\in A_{\bar{\nu}}^{\Omega}}\lambda_{\bar{\nu},\bar{m}}e^{\bar{\nu},\bar{m}}\, , 
\]
where $\{e^{\bar{\nu},\bar{m}}, \bar{\nu}\in \mathbb{N}_0^d,\ \bar{m}\in A_{\bar{\nu}}^{\Omega}\}$ is the canonical orthonormal basis of 
$s^{0,\Omega}_{2,2}f$. Obviously
\[
\| id^*-S_J: s^{t,\Omega}_{p_1,2}f \to s^{0,\Omega}_{p_2,2}f\|
\leq \sum_{\mu=J+1}^{\infty}\| id^*_{\mu}: (s^{t,\Omega}_{p_1,2}f)_{\mu} \to (s^{0,\Omega}_{p_2,2}f)_{\mu}\|\,.
\]
Lemma \ref{ba2} yields
\[
\| id^*-S_J: s^{t,\Omega}_{p_1,2}f \to s^{0,\Omega}_{p_2,2}f\|\leq \sum_{\mu=J+1}^{\infty}2^{-t\mu }
\lesssim 2^{-Jt} \, .
\] 
Because of $\text{rank}(S_J)\asymp 2^JJ^{d-1}$ we conclude in case $n=2^JJ^{d-1}$ that 
\[
a_n( id^*)\lesssim 2^{-Jt}\, .
\]
Now using the same argument as at the end of the proof of Theorem \ref{seq-1} and the inequality $x_n \leq a_n$ we get
$$x_n(id^*)\lesssim n^{-t}(\log n)^{(d-1)t}.$$
{\it Step 3.}  Estimate from above for the case $1\leq p_1< p_2< 2$ and $t>\frac{1}{p_1}-\frac{1}{p_2}$. By defining $\theta=\frac{1/p_1-1/p_2}{1/p_1-1/2}$  we obtain
\beqq
\theta\in (0,1)\qquad\text{and}\qquad\frac{1}{p_2}=\frac{1-\theta}{p_1}+\frac{\theta}{2}.
\eeqq
Let $t_1=\frac{1}{p_1}-\frac{1}{p_2}$ and $t_2=\frac{1}{2}-\frac{1}{p_2}$. Then the condition $t>\frac{1}{p_1}-\frac{1}{p_2}$ implies
\beqq
(1-\theta) t_1+ \theta t_2 =0,\qquad t-t_1>0\qquad\text{and}\qquad t-t_2>\frac{1}{p_1}-\frac{1}{2}.
\eeqq
This yields
\beqq
[s^{t_1,\Omega}_{p_1,2}f, s^{t_2,\Omega}_{2,2}f ]_{\theta}=s^{0,\Omega}_{p_2,2}\,,
\eeqq
see \cite[Thm. 4.6]{Vy}. Here by $[X,Y]_{\theta}$, $\theta\in (0,1)$, we denote the classical complex interpolation method of Calder\'on, see \cite{BL,lun,t78} for details. Employing the lifting property, see \cite[Lemma 7.3]{KiSi}, results in Step 2 and Theorem \ref{seq-1} we conclude that
\beq\label{int1}
x_n(id : s^{t,\Omega}_{p_1,2}f\to s^{t_1,\Omega}_{p_1,2}f)
&\asymp& x_n(id: s^{t-t_1,\Omega}_{p_1,2}f\to s^{0,\Omega}_{p_1,2}f)\nonumber\\
&\asymp& n^{-t+t_1}(\log n)^{(d-1)(t-t_1)}
\eeq
and
\beq\label{int2}
x_n(id : s^{t,\Omega}_{p_1,2}f\to s^{t_2,\Omega}_{2,2}f)
&\asymp& x_n(id: s^{t-t_2,\Omega}_{p_1,2}f\to s^{0,\Omega}_{2,2}f)\nonumber\\
&\asymp& n^{-t+t_2}(\log n)^{(d-1)(t-t_2)}.
\eeq
The interpolation property of the Weyl numbers, see Proposition \ref{inter}, results in
\beqq
x_{2n-1}(id^*: s^{t,\Omega}_{p_1,2}f\to s^{0,\Omega}_{p_2,2}f)\lesssim x_n^{1-\theta}(id :  s^{t,\Omega}_{p_1,2}f\to s^{t_1,\Omega}_{p_1,2}f) \cdot x_n^{ \theta}(id^ : s^{t,\Omega}_{p_1,2}f\to s^{t_2,\Omega}_{2,2}f). 
\eeqq
Inserting \eqref{int1} and \eqref{int2} into this inequality we complete the proof.
\end{proof}

\begin{theorem}\label{seq-3}
Let either $2\leq p_2<p_1<\infty$, $t>\frac{{1}/{p_2}-{1}/{p_1}}{{p_1}/{2}-1}$ or $2< p_1\leq p_2< \infty$, $t>\frac{1}{p_1}-\frac{1}{p_2}$. Then
$$x_n(id^*)\asymp  n^{-t+\frac{1}{p_1}-\frac{1}{p_2}}(\log n)^{(d-1)(t-\frac{1}{p_1}+\frac{1}{p_2})}\, ,\ \ \ n\geq 2.$$
\end{theorem}
\begin{proof}
{\it Step 1. }Estimate from below. Since $2<p_1 $ we choose $\epsilon >0 $ such that $2\leq p_1-\epsilon$. Then Lemma \ref{lowa} and \eqref{ba4-2} with $p_2\geq 2$ yield 
\beqq
x_n(id^*)
&\gtrsim & 2^{\mu(-t+\frac{1}{p_1}-\frac{1}{p_2})}x_n(id_{p_1-\epsilon,p_2}^{D_{\mu}} ) .
\eeqq
Now \eqref{wth1} and Lemma \ref{Weyl1} (ii) with $n=\big[\frac{D_{\mu}}{2}\big]$ imply $x_n(id_{p_1-\epsilon,p_2}^{D_{\mu}} )\asymp 1 $, which results in
$$x_n(id^*)\gtrsim 2^{\mu(-t+\frac{1}{p_1}-\frac{1}{p_2})}.$$
Because of $2^{\mu}\asymp \frac{n}{(\log n)^{d-1}}$ the desired estimate follows.\\
{\it Step 2. }Estimate from above. We use the diagram
\tikzset{node distance=4cm, auto}

\begin{center}
\begin{tikzpicture}
 \node (H) {$s^{t,\Omega}_{p_1,2}f$};
 \node (L) [right of =H] {$s^{0,\Omega}_{p_2,2}f $};
 \node (L2) [right of =H, below of =H, node distance = 2cm ] {$ s^{t,\Omega}_{p_1,p_1}f$};
 \draw[->] (H) to node {$id^*$} (L);
 \draw[->] (H) to node [swap] {$id^1$} (L2);
 \draw[->] (L2) to node [swap]{$id^2$} (L);
 \end{tikzpicture}
\end{center}
By Property (c) of the $s$-numbers we find
\beqq
x_n(id^*)\leq \|id^1\|\cdot x_n(id^2).
\eeqq
In \cite{KiSi} it has been proved
\beqq
x_n(id^2)\asymp n^{-t+\frac{1}{p_1}-\frac{1}{p_2}}(\log n)^{(d-1)(t-\frac{1}{p_1}+\frac{1}{p_2})} \,,\ \ \ n\geq 2.
\eeqq
This finishes the proof.
\end{proof}
\begin{theorem}\label{seq-4}
Let either $2\leq p_2<p_1<\infty$, $0<t<\frac{{1}/{p_2}-{1}/{p_1}}{{p_1}/{2}-1}$ or $1\leq p_2\leq 2<p_1< \infty$,  $0<t<\frac{1}{p_1}$. Then we have
$$ x_n(id^*)\asymp n^{-\frac{tp_1}{2}}(\log n)^{(d-1)\frac{tp_1}{2}} \, ,\  \ n\geq 2.$$
\end{theorem}
\begin{proof}{\it Step 1. }Estimate from below. \\
{\it Substep 1.1.} The case $1\leq p_2\leq 2<p_1<\infty$ and $0<t<\frac{1}{p_1}$. From Lemmas \ref{lowa} and \ref{low}  we have
\beqq
x_n(id^*)\gtrsim 2^{-t\mu},\ \ n=\big[\mu^{(d-1)}2^{\frac{2\mu}{p_1}}\Big]\, .
\eeqq
Rewriting this in dependence of $n$ we get 
\beqq
x_n(id^*)\gtrsim n^{-\frac{tp_1}{2}}(\log n)^{(d-1)\frac{tp_1}{2}}.
\eeqq
{\it Substep 1.2.} The case $2\leq p_2<p_1<\infty$ and $0<t<\frac{{1}/{p_2}-{1}/{p_1}}{{p_1}/{2}-1}$. We consider the commutative diagram 
\tikzset{node distance=4cm, auto}
\beqq
\begin{tikzpicture}
 \node (H) {$s^{t,\Omega}_{p_1,2}f$};
 \node (L) [right of =H] {$s^{0,\Omega}_{2,2}f $};
 \node (L2) [right of =H, below of =H, node distance = 2cm ] {$ s^{0,\Omega}_{p_2,2}f$};
 \draw[->] (H) to node {$id^1$} (L);
 \draw[->] (H) to node [swap] {$id^*$} (L2);
 \draw[->] (L2) to node [swap]{$id^2$} (L);
 \end{tikzpicture}
\eeqq
Property (c) of the Weyl numbers  yields
\beqq
x_n(id^1)\leq x_n(id^*)\cdot \|id_2\|.
\eeqq
Applying the result in Substep 1.1 with $p_2=2$ we obtain the desired estimate.  \\
{\it Step 2.} Estimate from above. \\
{\it Substep 2.1.} The case $2\leq p_2<p_1<\infty$ and $0<t<\frac{{1}/{p_2}-{1}/{p_1}}{{p_1}/{2}-1}$. Choosing $L=\big[J\frac{p_1}{2}\big]$, we obtain from \eqref{sum}
 \begin{equation}\label{sum2}
 x_n(id^*)\lesssim \sum_{\mu=J+1}^{L}x_{n_{\mu}}(id_{\mu}^*)+ 2^{-t\big[J\frac{p_1}{2}\big]} .
 \end{equation}
Next we define
\[
n_{\mu}:= \big[D_{\mu}\, 2^{\{(\mu-L)\beta+J-\mu \}}\big]\leq D_{\mu}\, , \qquad J+1 \le \mu \le L\, ,
\]
where $\beta >0$ will be fixed later on. Hence
\be\label{ws-27}
\sum_{\mu=J+1}^{L} n_{\mu} \lesssim 2^{J}\, J^{d-1}\, .
\ee
Employing Corollary \ref{cor-weyl} we obtain
\beqq
x_{n_{\mu}}(id_{\mu}^*)&\lesssim& 2^{\mu(-t+\frac{1}{p_1}-\frac{1}{p_2})}\Big(\dfrac{D_{\mu}}{n_{\mu}}\Big)^{\frac{1}{r}}\\
 &\lesssim& 2^{\mu(-t+\frac{1}{p_1}-\frac{1}{p_2})}2^{-\frac{(\mu-L)\beta+J-\mu}{r}}.
\eeqq
Recall $\frac{1}{r}=\frac{1/p_2-1/p_1}{1-2/p_1}$. Then the sum in \eqref{sum2} is estimated by
\beqq
\sum_{\mu=J+1}^{L}x_{n_{\mu}}(id_{\mu}^*)
&\lesssim &    \sum_{\mu=J+1}^{L} 2^{\mu(-t+\frac{1}{p_1}-\frac{1}{p_2}+\frac{1}{r}-\frac{\beta}{r})}2^{\frac{L\beta-J}{r}}.
\eeqq
Observe that the condition $t < \frac{{1}/{p_2}-{1}/{p_1}}{{p_1}/{2}-1}$ implies 
\[ 
 -t+\frac{1}{p_1}-\frac{1}{p_2}+\frac{1}{r}>0\,.\ \ 
 \]
Because of  this we can choose $\beta>0$ such that $-t+\frac{1}{p_1}-\frac{1}{p_2}+\frac{1}{r}-\frac{\beta}{r}>0$. Consequently we obtain
\beqq
\sum_{\mu=J+1}^{L}x_{n_{\mu}}(id_{\mu}^*)
&\lesssim &  2^{L(-t+\frac{1}{p_1}-\frac{1}{p_2}+\frac{1}{r}-\frac{\beta}{r})}2^{\frac{L\beta-J}{r}}\\
&\lesssim &  2^{L(-t+\frac{1}{p_1}-\frac{1}{p_2}+\frac{1}{r})}2^{\frac{-J}{r}}\, .
\eeqq
Replacing $L$ by $\big[\frac{p_1}{2}J\big]$,  a simple calculation leads to
$$\sum_{\mu=J+1}^{L}x_{n_{\mu}}(id_{\mu}^*)\lesssim  2^{-\frac{tp_1}{2}J}\,.$$
Hence, taking \eqref{sum2} and \eqref{ws-27} into account we find
$$x_{c.2^JJ^{d-1}}(id^*)\lesssim 2^{-\frac{tp_1}{2}J}.$$
Now, employing  the same argument as at the end of the proof of Theorem \ref{seq-1}, the claim follows. \\
{\it Substep 2.2.} The case $1\leq p_2\leq 2<p_1<\infty$ and $0<t<\frac{1}{p_1}$.
This time we employ the diagram
\tikzset{node distance=4cm, auto}
\beqq
\begin{tikzpicture}
 \node (H) {$s^{t,\Omega}_{p_1,2}f$};
 \node (L) [right of =H] {$s^{0,\Omega}_{p_2,2}f $};
 \node (L2) [right of =H, below of =H, node distance = 2cm ] {$ s^{0,\Omega}_{2,2}f$};
 \draw[->] (H) to node {$id^*$} (L);
 \draw[->] (H) to node [swap] {$id^1$} (L2);
 \draw[->] (L2) to node [swap]{$id^2$} (L);
 \end{tikzpicture}
\eeqq
The inequality 
\beqq
x_n(id^*)\leq x_n(id^1)\cdot \|id^2\|,
\eeqq
and
\beqq
x_n(id^1)\lesssim n^{-\frac{tp_1}{2}}(\log n)^{(d-1)\frac{tp_1}{2}}\, ,\ \ \ \ t<\frac{{1}/{2}-{1}/{p_1}}{{p_1}/{2}-1}=\frac{1}{p_1},
\eeqq
see Substep 2.1, yield the desired estimate. The proof is complete.
\end{proof} 
\begin{theorem}\label{seq-5}
Let $1\leq p_2\leq 2<p_1<\infty$ and $t>\frac{1}{p_1}$. Then we have
$$x_n(id^*)\asymp n^{-t+\frac{1}{p_1}-\frac{1}{2}}(\log n)^{(d-1)(t-\frac{1}{p_1}+\frac{1}{2})}\, \ \ n\geq 2.$$
\end{theorem}
\begin{proof}
{\it Step 1.} Estimate from below. Because $2<p_1$ we choose $\epsilon >0 $ such that $2< p_1-\epsilon$. Then Lemma \ref{lowa} and \eqref{ba4-2} yield
\beqq
x_n(id^*)&\gtrsim&\mu^{(d-1)(\frac{1}{2}-\frac{1}{p_2})}2^{\mu(-t+\frac{1}{p_1}-\frac{1}{p_2})}x_n(id_{p_1-\epsilon,p_2}^{D_{\mu}} )
\eeqq
Employing Lemma \ref{Weyl1} (iii) with $n=\big[\dfrac{D_{\mu}}{2}\big]$ we have found
\beqq
x_n(id^*)&\gtrsim&  \mu^{(d-1)(\frac{1}{2}-\frac{1}{p_2})}2^{\mu(-t+\frac{1}{p_1}-\frac{1}{p_2})}D_{\mu}^{\frac{1}{p_2}-\frac{1}{2}}\\
 &\asymp& 2^{\mu(-t+\frac{1}{p_1}-\frac{1}{2})}.
\eeqq
Because of $2^{\mu}\asymp\frac{n}{(\log n)^{d-1}}$ this implies the desired estimate. \\
{\it Step 2. }Estimate from above. The claim follows from the  results in Theorem \ref{seq-3} and the argument given in Substep 2.2 of the proof of Theorem \ref{seq-4}.
\end{proof}
\subsection{The results for Bernstein numbers}
Let us recall the behaviour of entropy numbers of the embeddings $ id^*: s_{p_1,2}^{t,\Omega}f \to s_{p_2,2}^{0,\Omega}f$, see \cite[Theorem 4.11]{Vy}. 
\begin{proposition}\label{entro} Let $1\leq p_1,p_2<\infty$ and $t>(\frac{1}{p_1}-\frac{1}{p_2})_+$. Then we have
\beqq
e_n(id^*)\asymp n^{-t}(\log n)^{(d-1)t}\, ,\ \ \ \ n\geq 2.
\eeqq
\end{proposition}
Now we are in position to prove the results for Bernstein numbers.
\begin{theorem}\label{conclud2}
Let $1< p_1,p_2< \infty $ and $ t>(\frac{1}{p_1}-\frac{1}{p_2})_+$. Then we have
$$b_n(id^*)\asymp n^{-\beta}(\log n)^{(d-1)\beta}\,,\ \ \ n\geq 2\,,$$
where $\beta$ given in Theorem \ref{main2}.
\end{theorem}
\begin{proof}
{\it Step 1.} Estimate from below. The lower estimates in the cases of high smoothness were carried out by Galeev \cite{Gale1}. However, for a better readability we give a proof here. We divide this step into some cases.
\begin{enumerate}
\item The case $p_1\leq p_2$. We have
 \beq\label{low00}
  b_{D_{\mu}}(id_{\mu}^*: (s^{t,\Omega}_{p_1,2}f)_{\mu}\to (s^{0,\Omega}_{p_2,2}f)_{\mu})
  &=& \inf_{\lambda\in (s^{t,\Omega}_{p_1,2}f)_{\mu}, \lambda\not=0}\frac{\|\lambda|(s^{0,\Omega}_{p_2,2}f)_{\mu} \|}{\|\lambda|(s^{t,\Omega}_{p_1,2}f)_{\mu} \| }.
   \eeq
  Since $p_1\leq p_2$, Lemma \ref{ba2} (i) yields
   \beqq
   \|\lambda|(s^{t,\Omega}_{p_1,2}f)_{\mu} \| \lesssim 2^{t\mu} \|\lambda|(s^{0,\Omega}_{p_2,2}f)_{\mu} \| 
   \eeqq
   for all $\lambda\in (s^{t,\Omega}_{p_1,2}f)_{\mu}$. Inserting this into \eqref{low00} we find
   \beqq
   b_{D_{\mu}}(id^*)\gtrsim 2^{-t\mu}.
   \eeqq
   Now rewriting this in dependence on $n$ we have found the desired estimate.
\item The case $p_2\leq p_1\leq 2$. The proof is similar to Step 1 in the proof of Theorem \ref{seq-2}. Since $p_1,p_2\leq 2$, Lemma \ref{lowa} and \eqref{ba4-1} result in
$$b_n(id^*) \gtrsim \mu^{(d-1)(\frac{1}{2}-\frac{1}{p_2})}2^{\mu(-t+\frac{1}{2}-\frac{1}{p_2})}b_n(id_{2,p_2}^{D_{\mu}} ). $$
By choosing $n=\big[\dfrac{D_{\mu}}{2}\big]$, applying \eqref{th3}, we obtain
\beqq
b_n(id^*) 
&\gtrsim & \mu^{(d-1)(\frac{1}{2}-\frac{1}{p_2})}2^{\mu(-t+\frac{1}{2}-\frac{1}{p_2})}(D_{\mu})^{\frac{1}{p_2}-\frac{1}{2}}\\
&\asymp & 2^{\mu(-t)}.
\eeqq
This implies the estimate from blow.
\item The case $p_2\leq 2< p_1$, $t>\frac{1}{p_1}$. To obtain the lower estimate in this case  we combine Lemma \ref{lowa} with \eqref{ba4-2} and \eqref{th2}. The argument is similar to Step 1 in the proof of Theorem \ref{seq-5}.
\item The case $2\leq p_2<p_1$, $t>\frac{1/p_2-1/p_1}{p_1/2-1}$. This time we use  Lemma \ref{lowa}, \eqref{ba4-2} and \eqref{th1}. We follow the arguments as in Step 1 of the proof of Theorem \ref{seq-3} to obtain the desired estimate.
\item The cases $2\leq p_2<p_1$, $t<\frac{{1}/{p_2}-{1}/{p_1}}{{p_1}/{2}-1}$ and $p_2\leq 2<p_1 $,  $t<\frac{1}{p_1}$ can be treated as in  Step 1 in the proof of Theorem \ref{seq-4}, see Lemmas \ref{lowa} and \ref{low}.

\end{enumerate}
{\it Step 2.} Estimate from above. The polynomial behaviour of the Weyl numbers of the embedding $ id^*: s_{p_1,2}^{t,\Omega}f \to s_{p_2,2}^{0,\Omega}f$ together with Theorems \ref{seq-3}$-$\ref{seq-5} and Lemma \ref{bern-weyl0}
results in the upper estimate in the cases $\max(2,p_2)<p_1$, i.e., (ii), (iii), (iv). The upper bound in the other cases are obtained by applying Proposition \ref{entro} and Lemma \ref{bern-en}. The proof is complete.
\end{proof}
\section{Proofs}\label{sec-proof}
The following lemma  allows  us to shift the results obtained for the sequence spaces, to the situation of function spaces. Recall that $\omega_n$ is either $x_n$ or $b_n$.
\begin{lemma}\label{weyl}
Let $1< p_1<\infty $,\ $1\leq p_2<\infty$ and $t\in \re$. Then
\[
\omega_n( id^*: s_{p_1,2}^{t,\Omega}f \to s_{p_2,2}^{0,\Omega}f)\asymp \omega_n\big(id: S_{p_1,2}^t F(\Omega)\to S_{p_2,2}^0 F(\Omega)\big)
\]
holds for all $n \in \N$.
\end{lemma}

\begin{proof}We use the argument given in \cite{KiSi}, see also Vybiral \cite{Vy}.\\
{\em Step 1.} Let $E: F^{t}_{p_1,2} (0,1) \to F^{t}_{p_1,2} (\re) $ denote a linear and continuous extension operator. Here $F^{t}_{p_1,2}(\re)$ and $F^{t}_{p_1,2} (0,1)$ are the Lizorkin-Triebel spaces defined on $\re$ and on $(0,1)$. For existence of those operators we refer  to \cite[3.3.4]{Tr83} or \cite{Ry}. Without loss of generality we may assume that 
\[
\supp Ef \subset \bigcup_{{\bar{\nu} \in \N_0^d}, \, \,   \bar{m}\in A_{\bar{\nu}}^\Omega }  \supp \Psi_{\bar{\nu},\bar{m}}\, , 
\] 
see Subsection \ref{seq-def}, for all $f\in F^{t}_{p_1,2} (0,1)$.
By $
\ce_d := E \otimes \ldots \otimes E 
$
we denote the $d$-fold tensor product operator which maps the space $F^{t}_{p_1,2} (0,1) \otimes_{\alpha_{p}} \ldots \otimes_{\alpha_{p}} F^{t}_{p_1,2} (0,1) $ 
into the space
$F^{t}_{p_1,2} (\re) \otimes_{\alpha_{p}} \ldots \otimes_{\alpha_{p}} F^{t}_{p_1,2} (\re) $. Here $\alpha_p$ is $p$-nuclear tensor norm. It follows that $\ce_d$ is a linear and continuous extension operator. For the identities
\beqq
S_{p_1,2}^t F (\Omega)=F^{t}_{p_1,2} (0,1) \otimes_{\alpha_{p}} \ldots \otimes_{\alpha_{p}} F^{t}_{p_1,2} (0,1)
\eeqq
and 
\beqq
S_{p_1,2}^t F (\R)=F^{t}_{p_1,2} (\re) \otimes_{\alpha_{p}} \ldots \otimes_{\alpha_{p}} F^{t}_{p_1,2} (\re) 
\eeqq 
we refer to Theorem \ref{equal} and \cite{SUt}. Hence  $\ce_d \in \cl (S_{p_1,2}^t F (\Omega), S_{p_1,2}^t F (\R))$.
\\
{\em Step 2.} We consider the commutative diagram
\[
\begin{CD}
S^{t}_{p_1,2}F(\Omega) @ >\mathcal{E}_d >> S^{t}_{p_1,2}F(\mathbb{R}^d) @>\mathcal{W}>>s_{p_1,2}^{t,\Omega}f\\
@V id VV @. @VV id^*V\\
S_{p_2,2}^{0}F(\Omega) @ <R_{\Omega}<< S_{p_2,2}^{0}F(\mathbb{R}^d)@ <\mathcal{W}^* << s_{p_2,2}^{0,\Omega}f\,.
\end{CD}
\]
Here $R_\Omega$ means the restriction to $\Omega$. The mappings $\mathcal{W}$ and $\mathcal{W}^*$ are defined as 
\[ 
\mathcal{W} f := \, \Big( 2^{|\bar{\nu}|_1}\, \langle f, \, \Psi_{\bar{\nu}, \bar{k}}\rangle
\Big)_{\bar{\nu}\in \N_0^d, \, \bar{k} \in A^{\Omega}_{\bar{\nu}}}
\qquad \text{and}\qquad
\mathcal{W}^* \lambda := 
\sum_{\bar{\nu} \in \N_0^ d}
\sum_{\bar{k} \in A^{\Omega}_{\bar{\nu}}} \lambda_{\bar{\nu}, \bar{k}} \, \Psi_{\bar{\nu}, \bar{k}}\,.
\]
Now, the boundedness of $\ce_d, \mathcal{W}, \mathcal{W}^*, R_\Omega$ and property (c) yield
$\omega_n (id)\lesssim \omega_n (id^*)$.
A similar argument with a slightly modified diagram yields $\omega_n (id^*)\lesssim \omega_n (id)$ as well.
\end{proof}
\noindent
{\bf Proof of Theorem \ref{main1}}. The claims in Theorem \ref{main1} are consequences of Lemma \ref{weyl}, Theorem \ref{equal} and Theorems \ref{seq-1}$-$\ref{seq-5}.  \qed
\vskip 0.3cm
\noindent
\noindent
{\bf Proof of Theorem \ref{main2}}. Taking into account Lemma \ref{weyl}, Theorems \ref{equal} and  \ref{conclud2} the claims in Theorem   \ref{main2} follow. \qed
\vskip 0.3cm
\noindent
{\bf Proof of Theorem \ref{main3}}. 
{\em Step 1.} We prove $(i)$.\\
{\it Substep 1.1.} Estimate from above. Under the given restrictions there always exists some $r>\frac{1}{2}$ such that $t>r+\Big(\frac{1}{p}-\frac{1}{2}\Big)_+$.
We  consider the commutative diagram

\tikzset{node distance=4cm, auto}

\begin{center}
\begin{tikzpicture}
 \node (H) {$S^{t}_{p}H(\Omega)$};
 \node (L) [right of =H] {$L_\infty(\Omega)$};
 \node (L2) [right of =H, below of =H, node distance = 2cm ] {$S^{r}_{2}H(\Omega)$};
 \draw[->] (H) to node {$id_1$} (L);
 \draw[->] (H) to node [swap] {$id_2$} (L2);
 \draw[->] (L2) to node [swap] {$id_3$} (L);
 \end{tikzpicture}
\end{center}
The multiplicativity of the Weyl numbers yields 
\[
x_{2n-1} (id_1)\le x_n (id_2)\, x_n (id_3)\, .
\]
Because of the lifting property, see \cite{KiSi}, and Theorem \ref{main1} we have
\beqq
x_n(id_2) &\asymp& x_n(id:S^{t-r}_{p}H(\Omega)\to L_2(\Omega) )\\
 &\asymp&
\left\{
\begin{array}{lll}
n^{-t+r}(\log n)^{(d-1)(t-r )}  & \mbox{if}&   1 < p \le 2\, , \: 
t-r> \frac{1}{p}-\frac{1}{2}\, , 
\\
n^{-t+r +\frac{1}{p}-\frac{1}{2}} (\log n)^{(d-1)(t-r - \frac{1}{p}+\frac{1}{2} )}  
& \mbox{if}& 2< p < \infty\, , \: t-r> \frac{1}{p} \, .
\end{array}
\right.
\eeqq
Now, employing
\beqq
x_n(id_3 )\asymp n^{-r+\frac 12}(\log n)^{(d-1)r}\,,\ \ \ n\geq 2\, ,
\eeqq
see \cite{CKS,Te93} and \eqref{an-xn}, the claim follows.\\
{\em Substep 1.2.} Estimate from below. We use again the multiplicativity of the Weyl numbers, this time in connection with Lemma \ref{rs1}
\beqq
&& \hspace{-1.7cm}
x_{2n-1} (id: ~ S_{p}^t H(\Omega) \to L_2(\Omega)) 
\\
& \le & 
x_{n} (id:~S_{p }^t H(\Omega) \to L_\infty (\Omega)) \, 
x_{n} (id:~L_{\infty}(\Omega) \to L_2(\Omega)) 
 \\
& = &
x_{n} (id:~S_{p}^t H(\Omega) \to L_\infty (\Omega)) \, n^{-1/2} \pi_2(id:~L_{\infty}(\Omega) \to L_2(\Omega))\, \\
& = &
x_{n} (id:~S_{p}^t H(\Omega) \to L_\infty (\Omega)) \, n^{-1/2}\, .
\eeqq
Here we have employed $\pi_2(id:~L_{\infty}(\Omega) \to L_2(\Omega))=1$, see  \cite[Example 1.3.9]{Pi-87}. Since
\beqq
n^{\frac 12}\, x_{2n-1} (id: && \hspace{-0.7cm} S_{p }^t H(\Omega) \to L_2(\Omega))
\\
& \asymp &
\left\{
\begin{array}{lll}
n^{-t+\frac{1}{2}} (\log n)^{(d-1)t}  & \quad & \mbox{if}\quad 1 < p \le 2\, , \:t > \frac{1}{p} - \frac 12 \, , 
\\
n^{-t +\frac{1}{p}} (\log n)^{(d-1)(t - \frac{1}{p} + \frac 12 )}  
& \quad & \mbox{if}\quad 2 < p < \infty\, , \: t>\frac{1}{p} \, , 
\\
\end{array}
\right.
\eeqq
see Theorem \ref{main1}, this proves the claimed estimate from below.\\
{\it Step 2.} We prove (ii).\\
{\it Substep 2.1.} Estimate from above. We employ the chain of continuous embeddings
$$S^{t}_{p}H(\Omega)\rightarrow S^0_{1,2}F(\Omega)\rightarrow L_1(\Omega),$$
see \cite{KiSi} and Theorem \ref{equal}. Now Lemma \ref{weyl}, Theorems \ref{seq-2},  \ref{seq-4}, \ref{seq-5} and the abstract properties of Weyl numbers, see Section \ref{sec-pro}, yield the upper bound. \\
{\it Substep 2.2.} Estimate from below. We prove  the case $1< p\leq 2$ and $ t> 0$. There always exists 
a pair $(\theta,p)$ such that $ \theta\in (0,1)$, $1 < p_0 < p$ and
\[
\| f|L_{p_0}(\Omega)\| \leq \|f|L_1(\Omega)\|^{1-\theta}\, \|f|L_{p}(\Omega)\|^{\theta}\qquad   \text{for all}\quad  f\in L_{p}(\Omega).
 \]
Next we employ the interpolation property of the Weyl numbers, see Proposition \ref{inter}, and obtain
\beqq
x_{2n-1} (id :  S^{t}_{p  }H(\Omega)&\to& L_{p_0}(\Omega))\\ &\lesssim&
x_n^{1-\theta}(id: S^{t}_{p }H(\Omega)\to L_1(\Omega))\, \, x_n^{\theta}(id:S^{t}_{p }H(\Omega)\to L_{p}(\Omega)).
\eeqq
Note that  $1< p_0<p\leq 2$ and $ t> 0$ imply
\beqq
x_{n}(id :S^{t}_{p }H(\Omega)\to L_{p}(\Omega))&\asymp& x_{n}(id :S^{t}_{p  }H(\Omega)\to L_{p_0}(\Omega))\\
&\asymp& n^{-t}(\log n)^{(d-1)t},
\eeqq
see Theorem \ref{main1}. This leads to
$$ x_n(id: S^{t}_{p }H(\Omega)\to L_1(\Omega))\gtrsim n^{-t}(\log n)^{(d-1)t}.$$
The lower bounds in the remaining cases can be proved similarly.\qed
\vskip 3mm
\noindent
{\bf Acknowledgements:}\ The author would like to thank Professor Winfried Sickel for many valuable discussions about this work.

\end{document}